\documentclass[12pt,reqno]{article}
\usepackage[usenames]{color}
\usepackage{amssymb}
\usepackage{graphicx}
\usepackage{amscd}
\usepackage[colorlinks=true,
linkcolor=webgreen,
filecolor=webbrown,
citecolor=webgreen]{hyperref}

\definecolor{webgreen}{rgb}{0,.5,0}
\definecolor{webbrown}{rgb}{.6,0,0}

\usepackage{color}
\usepackage{fullpage}
\usepackage{float}

\usepackage{graphics,amsmath,amssymb}
\usepackage{amsthm}
\usepackage{amsfonts}
\usepackage{latexsym}

\begin{document}

\vspace*{2.1cm}

\theoremstyle{plain}
\newtheorem{theorem}{Theorem}
\newtheorem{corollary}[theorem]{Corollary}
\newtheorem{lemma}[theorem]{Lemma}
\newtheorem{proposition}[theorem]{Proposition}
\newtheorem{obs}[theorem]{Observation}
\newtheorem{claim}[theorem]{Claim}

\theoremstyle{definition}
\newtheorem{definition}[theorem]{Definition}
\newtheorem{example}[theorem]{Example}
\newtheorem{remark}[theorem]{Remark}
\newtheorem{conjecture}[theorem]{Conjecture}
\newtheorem{question}[theorem]{Question}

\begin{center}

\vskip 1cm

{\Large\bf Face $S_3$-colorings and Petrie monochromatic circuits for efficient total arc colorings of cubic maps of girth 4} %
\vskip 5mm
\large
Italo J. Dejter

University of Puerto Rico

Rio Piedras, PR 00936-8377

\href{mailto:italo.dejter@gmail.com}{\tt italo.dejter@gmail.com}
\end{center}


\begin{abstract} Let $G=(V(G),E(G))$ be a finite connected simple cubic graph. A total coloring $\mu:V(G)\cup E(G)\rightarrow\{0,1,2,3\} $ is efficient if each color yields an efficient dominating set, where the efficient dominating condition applies to the restriction of each color class to $V(G)$. If the girth of $G$ is 4 and $\mu$ uses 4 colors both on vertices and edges of each 4-cycle, then $\mu$ is an ETGC. Maps $M(G)$ of such $G$ are considered on the orientable surfaces. We show that a necessary condition for ETGCs to exist on $M(G)$ is that they have only face lengths divisible by 4. When such ETGCs exist, for which 7 constructive tools are provided, the faces of the maps $M(G)$ possessing them are shown to be colorable by the group $S_3$ of 3-permutations whenever the ETGCs are periodic. For the vertex coloring of each ETGC, there are two orthogonal edge colorings: directed and reversed. There are two injections into the set of face $S_3$-colorings: one from the set of periodic directed ETGCs, the other from that of periodic reversed ETGCs. We take each edge of $G$ as a pair of opposite arcs. Let $A^\leftrightarrow$ be the set of such arcs and let $G^\leftrightarrow$ be the digraph with vertex set $V(G)$ and arc set $A^\leftrightarrow$. A ditotal coloring $\mu^\leftrightarrow$ of $G^\leftrightarrow$ is given by having bijections via $\mu^\leftrightarrow$ from the sets formed by: {\bf(a)} each vertex and the arcs arriving to it; {\bf(b)} each vertex and the arcs departing from it; {\bf(c)} each vertex and its neighbors; {\bf(d)} each pair of oppositely oriented arcs of an edge and its end-vertices. Then, each ETGC of an $M(G)$ has a ditotal coloring. The Petrie dual of the resulting map $M(G^\leftrightarrow)$ is formed by monochromatic directed cycles.
 \end{abstract}

\section{Introduction}\label{s0}

\begin{definition}\label{tc} Given a simple graph $G$, a {\it total coloring} (or {\it TC}) of $G$ is a color assignment to the vertices and edges of $G$ such that no two incident or adjacent elements (vertices or edges)
are assigned the same color. The {\it total chromatic number} $\chi''(G)$ of $G$ is the least number of colors required by a TC of $G$.\end{definition}

A recent survey \cite{tc-as} contains an updated bibliography on TCs. The TC Conjecture, posed independently by Behzad \cite{B1,B2} and by Vizing \cite{V}, asserts that $\chi''(G)$ is either $\Delta(G)+1$ or $\Delta(G)+2$, where $\Delta(G)$ is the largest degree of any vertex of $G$. 
 The TC Conjecture was established for cubic graphs \cite{Mazzu,Feng,Rosen,Vi}, meaning that the total chromatic number of cubic graphs is either 4 or 5. To decide whether a cubic graph $G$ has total chromatic number $\Delta(G)+1$, even for bipartite cubic graphs, is NP-hard \cite{Arroyo}.

\subsection{Efficient total colorings}

In \cite{+1}, total colorings of $k$-regular graphs $G$ of girth $k+1$ ($1<k\in\mathbb{Z}$) in the presence of {\it efficient dominating sets} \cite{worst,Tomai,D73,Deng, EDS,Knor} were considered. Definitions~\ref{antes} and \ref{ahora}, below, are introduced in order to apply such presence further down. Efficient dominating sets are a particular case of {\it perfect dominating sets} \cite{Araujo,Horak,PDS,Delgado}, considered in particular in Section~\ref{Messi}.

\begin{definition}\label{antes} A vertex subset $\Sigma$ of $G$ is said a {\it perfect dominating set}, or {\it PDS}, of $G$ if each vertex of the complementary graph $G\setminus\Sigma$ is adjacent to just one element of $\Sigma$.
 Let $N_G[v]$ and $N_G(v)$ be the {\it closed} and {\it open neighborhoods} \cite{D73} of a vertex $v\in V(G)$, respectively.
If $\mu$ is a total coloring of $G$ and $S_i=\{v\in V(G):\mu(v)=i\}$, where $i\in[k]_0=\{0,1,\ldots,k\}$,
then $S_i$ is an {\it efficient domination set}, or {\it EDS}, of $G$, also called {\it perfect code}, if $|N_G[v]\cap S_i|=1,\forall v\in V(G),\forall i\in[k]_0$. Or equivalently: if each $S_i$ is independent and every vertex outside $S_i$ has exactly one neighbor in $S_i$.  Clearly, each $S_i$ is a PDS.
\end{definition}

\begin{definition}\label{ahora} 
A vertex and edge coloring $\mu$ of a $k$-regular simple graph $G$ ($2\le k\in\mathbb{Z}$) is said to be an {\it efficient total coloring} ({\it ETC}), and $G$ said to be {\it ETCed}, if:
\begin{enumerate}
\item[\bf(a)] $\mu$ is a TC as in Definition~\ref{tc}; in particular, each $v\in V(G)$ and its neighbors are assigned by $\mu$ all the colors in $[k]_0$ via a bijection $N_G[v]=N_G(v)\cup\{v\}\leftrightarrow[k]_0$;
\item[\bf(b)] $\mu$ partitions $V(G)$ into $k+1$ EDSs as in Definition~\ref{antes}.
\end{enumerate}\end{definition}

\noindent Definition~\ref{ahora} implies that the total chromatic number of $G$ is $\Delta(G)+1$. 

In the rest of this work, we deal with simple cubic graphs $G$ of girth 4. 
One of our purposes is to determine ETCs of such graphs. They yield {\it edge-girth colorings} (see Definition~\ref{egc}, below) in the prisms $G\square K_2$, where $K_n$ is the complete graph on $n$ vertices and $G\square G'$ stands for the Cartesian product of two graphs $G$ and $G'$ \cite[pg. 30]{Bondy}. 
  
 \begin{definition}\label{egc} Let $G$ be a simple cubic graph of girth 4.
An {\it edge-girth coloring} (or {\it EGC}) of $G$ is a proper edge coloring by means of four colors, each girth cycle colored with the four colors, each color used precisely once. 
\end{definition}

In order to present our results, we need every girth cycle $C$ of $G$ to be colorable with four colors, each color used exactly once on the vertices of $C$ and exactly once on the edges of $C$. This leads  to the following Definition~\ref{hoy}, in which, in addition, the total coloring of the girth cycles is combined with the concept of ETC in Definition~\ref{ahora}. 

\begin{definition}\label{hoy} An {\it efficient total girth coloring} ({\it ETGC}) of a simple cubic graph $G$ of girth 4 is an ETC of $G$ in which each girth cycle of $G$ is colored with four colors, each color used just once on vertices and just once on edges. 
\end{definition}

The constructions in the following sections involve the existence of ETGCs in simple cubic graphs $G$ of girth 4.
 Those constructions involve the existence of cubic graphs $G$ of corresponding EGCs on the prisms $G\square K_2$, motivating the following definition.

\begin{definition}\label{ortho}
Given a simple graph $G$ and a TC $C$ of $G$, if $C$ is extensible to two ETCs $C'$ and $C''$ with differing colors on every edge of $G$, then $C'$ and $C''$ are said to be {\it mutually orthogonal} ETCs and their induced edge colorings are said to be {\it mutually orthogonal}, too.
\end{definition}     

 \subsection{Maps, belts, zonogons and cutouts}
 
\begin{definition}\label{belt}
The {\it genus} of a graph $G$ is the minimum 
number $g=g(G)$ of handles to be added to a sphere of $\mathbb{R}^3$ in order to avoid edge crossings in any drawing of $G$ in the resulting connected closed orientable surface $S$. Let $G$ be a connected simple cubic graph of girth 4 and genus $g=g(G)$.
 A {\it map} $M(G)$ of $G$ is an embedding of $G$ into $S$ (whose {\it genus} is $g(S)=g=g(G)$) such that the connected components of $S\setminus G$
 are homeomorphic to open disks $D$ in $\mathbb{R}^2$ that we call {\it open face}. The closure of each such $D$ is said to be a {\it (closed) face} of $M(G)$, and $G$ the {\it 1-skeleton} of $M(G)$.   
A {\it belt} of $M(G)$ is a face-boundary cycle of $M(G)$. An $\ell$-{\it belt} of $M(G)$ is a belt of length $\ell$, where $3<\ell\in\mathbb{Z}$, i.e. with $\ell$ vertices and $\ell$ edges. 
The {\it Euler characteristic} $\chi=\chi(M(G))$ of $M(G)$ (in $S$) is $\chi(M(G))=|V(G)|-|E(G)|+|F(M(G))|=2-2g(G)$, where $|\cdot|=$ cardinality and $F(M(G))$ is the number of faces of $M(G)$ \cite[pg. 278]{Bondy}. In topological graph theory \cite{GrossT}, $M(G)$ is represented by a {\it rotation system}, that is: an assignment of a cyclic permutation $p_v=(e_1\;e_2\;e_3)$ of the edges $e_i$ ($i=1,2,3$) incident to each $v\in V(G)$ yielding the clockwise (or counterclockwise) order of those edges around $v$ in $S$.
\end{definition}
  
\begin{definition}\label{cutout} Let $1<n\in\mathbb{Z}$.
A {\it $2n$-zonogon} \cite[pg. 319]{BV} in $\mathbb{R}^2$ is a centrally symmetric convex polygon $X=(L_1,L_2,\ldots,L_n,L_1^{-1},L_2^{-1},\ldots,L_n^{-1})$ of $2n$ oriented sides composed by antipodal pairs $\{L_i,L_i^{-1}\}$ of parallel sides with opposite orientations, where each clockwise oriented side $L_i$ and corresponding counterclockwise oriented side $L_i^{-1}$ have equal lengths, for $i=1,\ldots,n$. A surface $S$ of genus $g=\lfloor\frac{n}{2}\rfloor$ is obtained by identification or ``gluing" or ``cancelling" of each side $L_i$ with its corresponding antipodal side $L_i^{-1}$. A map $M(G)$ of genus $g=\lfloor\frac{n}{2}\rfloor$ is obtained in $S$,
where each identification $L_i\equiv L_i^{-1}$ is said to be a {\it sides-cancelling segment}. Each map $M(G)$ of genus $g=\lfloor\frac{n}{2}\rfloor$ is representable as an {\it $n$-cutout} $X$, namely a $2n$-zonogon from which $M(G)$ is recoverable by the said sides identifications.
 In particular, a 2-cutout is also said to be a {\it bicutout}, in which case $L_1,L_1^{-1}$ are assumed horizontal and $L_2,L_2^{-1}$ vertical, (though in Section~\ref{Messi} these directions of bicutouts are modified).
 \end{definition}

\begin{definition}\label{cut}
If $G$ is planar, then
a {\it cutout} of $G$ is defined as a closed rectangle $X=[0,x_0]\times[0,y_0]\subset\mathbb{R}^2$ with $(x_0,y_0)\in V(G)\subset X\cap\mathbb{Z}^2$ and $E(G)$ given by arcs in $X$ with ends in $\mathbb{Z}^2$ but not crossing each other in $\mathbb{R}^2\setminus\mathbb{Z}^2$. 
Vertical (resp. horizontal) segment identification $\{0\}\times[0,y_0]\equiv\{x_0\}\times[0,y_0]$ (resp. $[0,x_0]\times\{0\}\equiv[0,x_0]\times\{y_0\}$) in $\mathbb{R}^2$ yields a  representation of $G$ with point identifications $(0,y)\equiv(x_0,y)$, for $0\le y\le y_0$ (resp. $(x,0)\equiv(x,y_0)$, for $0\le x\le x_0$), said to be {\it identifications in parallel}. Such cutout is said to be an {\it $(x_0\times y_0)$-xcutout} (resp. {\it $(x_0\times y_0)$-ycutout}).
\end{definition}

\begin{definition}\label{25}
Let $G$ be a connected simple cubic graph of girth 4 and let $v\in V(G)$. Then, either $v$ belongs to a 4-cycle of $G$, in which case $v$ is said to be an {\it O-vertex}, or not, in which case $v$ is said to be a {\it Y-vertex}.
Let $e\in E(G)$. If the end-vertices of $e$ are Y-vertices, then $e$ is said to be a {\it Y-edge}. 
If $G$ has Y-vertices only as end-vertices of Y-edges, or has no Y-vertices at all, then $G$ is said to be a {\it Y-graph}. 
Let $M(G)$ be a map of a Y-graph $G$. The edges of a Y-graph $G$ not in 4-belts of $M(G)$ are said to be {\it H-edges}; they form a subset $E_H(G)$ of $E(G)$. The components of $G\setminus E_H(G)$ are said to be the {\it clusters} of $G$. 
A $2\ell$-belt $B$ of $M(G)$ with associated cycle $(e_1,e_2,\ldots,e_{2\ell-1},e_{2\ell})$ of subsequently adjacent edges $e_i$ ($i=1,\ldots,2\ell$) is said to be a {\it Y-belt} if it has Y-edges, if any, only in pairs $(e_h,e_{h+\ell})$ of opposite edges of $B$, where $1\le h\le\ell$. 
If $M(G)$ has only belts of even lengths and all its belts are Y-belts, then $M(G)$ is said to be a {\it Y-map}. 
\end{definition}

\begin{remark}\label{ndrangheta}
The maps $M(G)$ are required to be Y-maps as in Definition~\ref{25} in order to avoid maps 
with 4-belts distributed too sparse locally, which is the case of the {\it toroidal} (i.e., $g=1$) map $M(G)$ of the leftmost bicutout in (\ref{oct}) of Example~\ref{contrera}, in which case there are not ETGCs in the corresponding graph $G$. Restricting from any maps to just Y-maps ensures that faces split by cutout borders are recoverable by identifications of opposite sides. 
\end{remark}

\begin{example}\label{contrera}
An example of a $(4\times 4)$-bicutout $X$ of a toroidal map $M(G)$ having an 8-belt $B$ with three Y-edges and no corresponding opposite Y-edges is given on the left side of (\ref{oct}), thus not satisfying the requirement
of Definition~\ref{25}. In the figure, O-vertices and Y-vertices are indicated by $\circ$ and $\bullet$, respectively, while edges are given by horizontal and vertical short segments. Missing edges in the border of $X$ are given as horizontal or vertical pairs of dots, i.e., diaereses and colons.
This $M(G)$ has twelve O-vertices, four Y-vertices, four 4-belts and four 8-belts, but $G$ does not admit ETGCs. 
By stacking copies of $X$ horizontally and vertically 
(see Section~\ref{extensions}), 
$(4\ell\times 4k)$-bicotouts of toroidal maps with similar inexistence are obtained. 

\begin{eqnarray}\label{oct}\begin{array}{cc|cc}
\circ\hspace*{1.2mm}-\hspace*{1.2mm}\circ\hspace*{1.2mm}-\hspace*{1.2mm}\circ\hspace*{2.2mm}..\hspace*{2.2mm}\circ\hspace*{1.2mm}-\hspace*{1.2mm}\circ&&&0\hspace*{2.2mm} _-^2\hspace*{2.2mm} 1\hspace*{2.2mm} _-^3\hspace*{2.2mm} 2\hspace*{2.2mm} _-^0\hspace*{2.2mm} 3\hspace*{2.2mm} _-^1\hspace*{2.2mm} 0\\
:\hspace*{17.5mm}|\hspace*{8.5mm}|\hspace*{6.5mm}:&&&\!_3|\hspace*{6mm}_0|\hspace*{6mm}_1|\hspace*{6mm}_2|\hspace*{6mm}_3|\\
\bullet\hspace*{1.2mm}-\hspace*{1.2mm}\bullet\hspace*{1.2mm}-\hspace*{1.2mm}\bullet\hspace*{1.2mm}-\hspace*{1.2mm}\bullet\hspace*{1.2mm}-\hspace*{1.0mm}\bullet&&&
2\hspace*{2.2mm}_1^-\hspace*{2.2mm}3\hspace*{2.2mm} _2^-\hspace*{2.2mm} 0\hspace*{2.2mm} _3^-\hspace*{2.2mm} 1\hspace*{2.2mm} _0^-\hspace*{2.2mm} 2
\\
\,|\hspace*{8.5mm}|\hspace*{27.0mm}|&&&\\
\circ\hspace*{6.9mm}\circ\hspace*{1.2mm}-\hspace*{1.2mm}\circ\hspace*{1.2mm}-\hspace*{1.2mm}\circ\hspace*{1.2mm}-\hspace*{1.0mm}\circ&&&\mbox{Two ortogonal ETGCs}\\
\,|\hspace*{8.5mm}|\hspace*{8.5mm}|\hspace*{8.5mm}|\hspace*{7.5mm}|&&&\\
\circ\hspace*{6.9mm}\circ\hspace*{1.2mm}-\hspace*{1.2mm}\circ\hspace*{6.6mm}\circ\hspace*{1.2mm}-\hspace*{1.0mm}\circ&&&0\hspace*{2.2mm} _-^3\hspace*{2.2mm} 1\hspace*{2.2mm} _-^0\hspace*{2.2mm} 2\hspace*{2.2mm} _-^1\hspace*{2.2mm} 3\hspace*{2.2mm} _-^2\hspace*{2.2mm} 3\\
\,|\hspace*{8.5mm}|\hspace*{8.5mm}|\hspace*{8.5mm}|\hspace*{7.5mm}|&&&\!_1|\hspace*{6mm}_2|\hspace*{6mm}_3|\hspace*{6mm}_0|\hspace*{6mm}_1|\\\
\circ\hspace*{1.2mm}-\hspace*{1.2mm}\circ\hspace*{1.2mm}-\hspace*{1.2mm}\circ\hspace*{2.2mm}..\hspace*{2.2mm}\circ\hspace*{1.2mm}-\hspace*{1.2mm}\circ&&&2\hspace*{2.2mm} _0^-\hspace*{2.2mm}3\hspace*{2.2mm} _1^-\hspace*{2.2mm} 0\hspace*{2.2mm} _2^-\hspace*{2.2mm} 1\hspace*{2.2mm} _3^-\hspace*{2.2mm} 2 \\
\end{array}\end{eqnarray}
\end{example}

\subsection{Belts of lengths divisible by 4}

\begin{theorem}\label{fo}
Let $G$ be a connected simple cubic graph of girth 4 and let $M(G)$ be a Y-map with an ETGC $\mu$ of four colors. Then,  $|V(G)|\equiv 0 \mod 4$.  Moreover,
$M(G)$ has only $\ell$-belts with $\ell\equiv 0 \mod 4$.
\end{theorem}

\begin{proof}\label{lim} A partition of $V(G)$ is given by
the closed neighborhoods around the vertices of a fixed EDS $\zeta$ that participates of $\mu$ as one of its four color classes, 
implying that each color class of $\mu$, among which $\zeta$, has size $V(G)/4$. This gives the necessary condition $|V(G)|\equiv 0 \mod 4$. The neighborhood of any fixed 4-belt in $M(G)$ has a distribution of colors of $\mu$ as in one of the two sides of (\ref{vac}), where $a,b,c,d\in[3]_0=\{0,1,2,3\}$ are pairwise different vertex colors that are shown in normal size, while edge colors are shown in subindex size.
\begin{eqnarray}\label{vac}\begin{array}{|llllll|}\hline
&d\hspace*{2.2mm}_-^c\hspace*{2.2mm}b\hspace*{2.2mm}_-^a\hspace*{2.2mm}c\hspace*{2.2mm}_-^d\hspace*{2.2mm}a
&&&&d\hspace*{2.2mm}_-^a\hspace*{2.2mm}b\hspace*{2.2mm}_-^d\hspace*{2.2mm}c\hspace*{2.2mm}_-^b\hspace*{2.2mm}a\\
\mbox{either}&\hspace*{8.5mm}\!_d|\hspace*{6mm}_b|&&\mbox{or}&&\hspace*{8.5mm}\!_c|\hspace*{6mm}_a|\\
&c\hspace*{2.2mm}_-^b\hspace*{2.2mm}a\hspace*{2.2mm}_-^c\hspace*{2.2mm}d\hspace*{2.2mm}_-^a\hspace*{2.2mm}b
&&&&c\hspace*{2.2mm}_-^d\hspace*{2.2mm}a\hspace*{2.2mm}_-^b\hspace*{2.2mm}d\hspace*{2.2mm}_-^c\hspace*{2.2mm}b\\\hline
\end{array}\end{eqnarray} 
It follows that any color continuation of $\mu$ along an $\ell$-belt $B$ sharing an edge with a 4-belt, e.g. the fourth vertical edge $c\,_\rightarrow^a\,d$ on the right of (\ref{vac}) seen as a downward arc from $c$ to $d$, has a subpath pattern of length 4 in $B$ as in (\ref{reverse}), where $x\in[3]_0$ and $x\notin\{b,c\}\subset[3]_0$.
\begin{eqnarray}\label{reverse}
\;a\,_\rightarrow^b\,c\,_\rightarrow^a\,d\,_\rightarrow^c\,b\,_\rightarrow^x\hspace{3mm}\mbox{ (or in reverse: }\hspace{3mm}
\;_\leftarrow^x\,a\,_\leftarrow^b\,c\,_\leftarrow^a\,d\,_\leftarrow^c\,b\;\;).\end{eqnarray}
 A subpath pattern as in (\ref{reverse}) will be said to be an {\it O-arc}.
We also consider the pattern in (\ref{reverse}) with at least one of the two central vertices being a Y-vertex, in which case the subpath pattern is said to be a {\it Y-arc}. This Y-arc is joined together with an opposite Y-arc as in (\ref{biarc}).
\begin{eqnarray}\label{biarc}
\;a\,_\rightarrow^b\,d\,_\rightarrow^a\,c\,_\rightarrow^d\,b\,_\rightarrow^x\hspace{3mm}\mbox{ (or in reverse: }\hspace{3mm}
\;_\leftarrow^x\,a\,_\leftarrow^b\,d\,_\leftarrow^a\,c\,_\leftarrow^d\,b\;\;).\end{eqnarray}
The opposition of Y-arcs in (\ref{reverse}) and (\ref{biarc}) means they share only their central edge but with opposite orientations,
thus forming an ``H'' shaped subgraph said to be a {\it bi-arc}, which is composed by 6 vertices and 6 arcs, the two central arcs forming the central edge.
If $G$ is a Y-graph, then each belt $B$ of $M(G)$ has the image $\mu(B)$ under $\mu$ splittable into O-arcs, thus having $|V(B)|/4$ vertices, resp. $|E(B)|/4$ edges, in each of the four colors, so
$M(G)$ has only $\ell$-belts with $\ell\equiv 0 \mod 4$. Otherwise, $G$ contains 
bi-arcs. In this case, a similar argument is valid for each of the belts of $M(G)$.
\end{proof}

The colors of $[3]_0$ for vertices and edges in all figures below are taken as in  (\ref{colors}):
\begin{eqnarray}\label{colors}\mbox{0=hazel, 1=red, 2=blue, 3=green, (auxiliary: light-gray, light blue, yellow).}\end{eqnarray}

\begin{remark}
The Y-map $M(G)$ represented via the leftmost ycutout of Figure~\ref{bluyelow} (Example~\ref{tor}) has two 8-belts separated by two vertical red edges. These two edges are the central edges of corresponding bi-arcs
 \begin{eqnarray}\label{and1}(\;1\,_\rightarrow^3\,0\,_\rightarrow^1\,2\,_\rightarrow^0\,3\,_\rightarrow^2)\cup(\;1\,_\rightarrow^3\,2\,_\rightarrow^1\,0\,_\rightarrow^2\,3\,_\rightarrow^0),\end{eqnarray} located around the ``short'' central edge, and
\begin{eqnarray}\label{and2}(\;1\,_\rightarrow^3\,2\,_\rightarrow^1\,0\,_\rightarrow^2\,3\,_\rightarrow^0)\cup(\;1\,_\rightarrow^3\,0\,_\rightarrow^1\,2\,_\rightarrow^0\,3\,_\rightarrow^2),\end{eqnarray} located around  the ``long'' central edge represented twice on left and right (before ycutout identification). The remaining belts of $M(G)$ are eight 4-belts. These, together with the two 8-belts separated by bi-arcs expressed in (\ref{and1}) and (\ref{and2}), receive the colors of $[3]_0$ for an ETGC $\mu$. That is also the case of the second ycutout in the figure, equivalent to the first one. The third ycutout has six 8-cycles and twelve 4-cycles and admits a similar treatment. The fourth ycutout has four 8-cycles, ten 4-cycles and two vertices whose 3 neighbors are O-vertices, namely the central blue vertex (color 2) and lower hazel vertex (color 0) twice represented on left and right (before identification).
\end{remark}

\begin{remark}
The {\it 2-toroidal} (i.e., $g=2$) Y-map obtained from the 8-cutout in Figure~\ref{paquita} (Examples~\ref{satan}-\ref{hoyo}-\ref{lastrem}-\ref{negating}) by gluing opposite sides has 14 belts denoted by the capital letters from A to N. It has six 4-belts A,B,C,D,E,F comprising 24 O-vertices, the 8-belt N formed by 8 Y-vertices and 8 Y-edges, the remaining 8-belts G,H,I,J,K,L, and the 16-belt M that shares one edge with each 4-belt and two opposite edges with N so that the end-vertices of the resulting 8 edges are the 16 vertices of M. For example, the following O-arcs are obtained from the figure:
$$\begin{array}{ccc}
A(G:N\;[2\,_\rightarrow^3\,1\,_\rightarrow^2\,0\,_\rightarrow^1\,3\,_\rightarrow^0]\;N)/(K,J)&&E(I:C\;[3\,_\rightarrow^1\,0\,_\rightarrow^3\,2\,_\rightarrow^0\,1\,_\rightarrow^2]\;D)/(H,M)\\
A(M:N\;[1\,_\rightarrow^0\,2\,_\rightarrow^1\,3\,_\rightarrow^2\,0\,_\rightarrow^3]\;C)/(J,K)&&E(L:N\;[0\,_\rightarrow^2\,3\,_\rightarrow^0\,1\,_\rightarrow^3\,2\,_\rightarrow^1]\;D)/(M,H)\\
B(G:N\;[2\,_\rightarrow^3\,1\,_\rightarrow^2\,0\,_\rightarrow^1\,3\,_\rightarrow^0]\;N)/(L,J)&&F(I:D\;[3\,_\rightarrow^1\,0\,_\rightarrow^3\,2\,_\rightarrow^0\,1\,_\rightarrow^2]\;C)/(H,M)\\
B(M:N\;[1\,_\rightarrow^0\,2\,_\rightarrow^1\,3\,_\rightarrow^2\,0\,_\rightarrow^3]\;D)/(J,L)&&F(K:N\;[0\,_\rightarrow^2\,3\,_\rightarrow^0\,1\,_\rightarrow^3\,2\,_\rightarrow^1]\;C)/(M,H)\\
\end{array}$$
where a typical O-arc along a belt $Y$ is encoded as $X(Y:Z\;[a\,_\rightarrow^b\,c\,_\rightarrow^a\,d\,_\rightarrow^c\,b\,_\rightarrow^x]W)\;/(U,V)$, sharing the arc $c\,_\rightarrow^a\,d$ with the 4-belt $X$, the arc $a\,_\rightarrow^b\,c$ with the belt $U$ and the arc $d\,_\rightarrow^c\,b$ with the belt $V$. Initial endvertex of $X$ is in $Z$ and its terminal endvertex in $W$.
\end{remark}

\begin{definition}\label{periodic}
An ETGC $\mu$ of a Y-map $M(G)$ as in Theorem~\ref{fo} satisfying (\ref{reverse}) with $x=d$ all along its belts is said to be {\it periodic}. Otherwise, $\mu$ is said to be {\it aperiodic}.
\end{definition}

Most examples of ETGCs below are periodic, 
with every vertex being either $O$-vertex or end-vertex of the central edge of a bi-arc, as on the map on the right of Figure~\ref{ahiva}, represented via a cutout. In such cases, face colorings via the group $S_3$ of 3-permutations, as introduced in Subsection~\ref{face}, are available.

\begin{example}\label{contra} 
There are Y-maps $M(G)$ with aperiodic ETGCs satisfying (\ref{reverse}) all along its belts but not always with $x=d$, as in the 8-belt (\ref{nelly1})
 of the bicutout $Y$ represented on the right of (\ref{te}) and in the 16-belt (\ref{nelly2})
of the bicutout $\Upsilon$ on the right of Figure~\ref{20-gray}, this one starting at its leftmost vertex and in the clockwise direction (namely to the left).

\begin{eqnarray}\label{nelly1}((1\,_\rightarrow^0\,2\,_\rightarrow^1\,3\,_\rightarrow^2\,0\,_\rightarrow^1)(\;2\,_\rightarrow^3\,1\,_\rightarrow^2\,0\,_\rightarrow^1\,3\,_\rightarrow^2))=(O_0O_0').\end{eqnarray}
\begin{eqnarray}\label{nelly2}
((3\,_\rightarrow^1\,2\,_\rightarrow^3\,0\,_\rightarrow^2\,1\,_\rightarrow^0)(3\,_\rightarrow^1\,2\,_\rightarrow^3\,0\,_\rightarrow^2\,1\,_\rightarrow^3)(2\,_\rightarrow^0\,3\,_\rightarrow^2\,1\,_\rightarrow^3\,0\,_\rightarrow^1)^2(_\leftarrow^0\,1\,_\leftarrow^2\,0\,_\leftarrow^1\,3\,_\leftarrow^0\,2)^{-1})=(O_1O_2O_3^2O_4^{-1}).
\end{eqnarray}

Note that (\ref{nelly1}) and (\ref{nelly2}) are decomposed into parenthesized O-arcs $O_0,O_0',O_1,O_2,O_3,O_4$, where exponent 2 indicates that such O-arc is concatenated with itself, and exponent $-1$ indicates that the inverse of the enclosed O-arc is meant to be in its place in the cycle.
The remaining pair of 8-belts on the right of Figure~\ref{20-gray} are decomposable as in  (\ref{nelly3}).
\begin{eqnarray}\label{nelly3}((0\,_\rightarrow^1\,3\,_\rightarrow^0\,2\,_\rightarrow^3\,1\,_\rightarrow^2)(_\leftarrow^2\,1\,_\leftarrow^0\,3\,_\leftarrow^1\,2\,_\leftarrow^3\,0)^{-1})
\;\;\mbox{ and }\;\;((2\,_\rightarrow^3\,1\,_\rightarrow^2\,0\,_\rightarrow^1\,3\,_\rightarrow^0)(_\leftarrow^0\,3\,_\leftarrow^2\,1\,_\leftarrow^3\,0\,_\leftarrow^1\,2)^{-1}).
\end{eqnarray}
On the other hand, in Figure~\ref{bluyelow} and Example~\ref{tor}, cycles of lengths $\ell\equiv 2 \mod 4$ exist in Y-maps $M(G)$ but they do not occur as $\ell$-belts, for they have at least two adjacent edges in common with some 4-belt, or larger $\ell$-belt, ($4<\ell\equiv 0\mod 4$).
\end{example}

\section{Propagation}\label{planar}

\begin{definition}\label{sabado}
Let $G$ be a cubic graph of girth 4 and let $X$ be a cutout of a Y-map $M(G)$. Given a vertex $v$ of $G$ incident to edges $e=(v,v_e),f=(v,v_f)$ and $g=(v,v_g)$, the {\it spray}  by color set $\{c_0,c_1,c_2,c_3\}=[3]_0$ at the vertex $v$ is the ETGC propagation rule given by:
\begin{enumerate}\item if $v,e,f$ are attributed colors $c_0,c_1,c_2$, respectively, then $g$ is assigned color $c_3$;
\item if in addition $v_e$ and $v_f$ are attributed colors $d_1$ and $d_2$, respectively, where $[3]_0=\{c_0,d_1,d_2,d_3\}$, then $v_g$ is assigned color $d_3$.
\end{enumerate}
\end{definition}

\begin{theorem}\label{tt} Let $G$, $M(G)$, $X$, $v$, $e$, $f$ and $g$ be as in Definition~\ref{sabado}.
Initializing the spray by coloring $\{v,e,f,g\}$ in bijection with $[3]_0$ and continuing via items 1-2 over the rest of $M(G)$, if an ETGC $\mu$ is obtained and If (\ref{reverse}) occurs always with $x=d$, then $\mu$ is periodic.
\end{theorem}

\begin{proof} 
 Let $\mathcal{B}=(B_1,B_2,\ldots,B_k)$ be a sequence composed by the Y-belts and 4-belts of $M(G)$. Let us have: {\bf(a)} the Y-belts, while remaining in the formation of $\mathcal{B}$ from left to right, preceding any remaining 4-belts; {\bf(b)} the 4-belts in each cluster (as in Definition~\ref{25}) forming a subsequence of contiguous terms in $\mathcal{B}$; {\bf(c)}
each belt $B_i\in\mathcal{B}$ with $i>1$ related to a previous belt $B_j\in\mathcal{B}$ ($j<i$) by either {\bf(i)}: sharing an edge $(v,u)\in E(B_i)$ with $B_j$, or {\bf(ii)}: reaching $B_j$ by means of an edge with end-vertices $v\in V(B_i)$ and $u\in V(B_j)$; {\bf(d)} any remaining cases of item {\bf(i)} preceding in $\mathcal{B}$ any remaining cases of item {\bf(ii)}. We adopt subsequent spraying always in the clockwise direction and start the spray from $B_1$ on. After having sprayed $B_1,B_2,\ldots,B_{i-1}$, we proceed inductively by spraying the vertices of $B_i$ starting from either the vertex $v$ or $u$ (as in items {\bf(i)} and {\bf(ii)}), one of which was not sprayed yet. 
The formation of O-arcs and Y-arcs is ensued. In the process, assume they can be enlisted parenthesized in terms as in (\ref{reverse}) and (\ref{biarc}), while taking $x=d$ whenever possible, and $x\ne d$ otherwise.
Then, spray following the order of Y-belts and 4-belts established by $\mathcal{B}$
provides each belt of $M(G)$ with a color cycle formed as a concatenation of O-arcs and Y-arcs, in which case $M(G)$ receives an ETGC $\mu$. If all belts color cycles admit $x=d$ in (\ref{reverse}), then
 each belt in $X$ receives colors periodically, as in the color cycle 
\begin{eqnarray}\label{sche}(a\,_-^b\,c\,_-^a\,d\,_-^c\,b\,_-^d\cdots\cdots a\,_-^b\,c\,_-^a\,d\,_-^c\,b\,_-^d\cdots\cdots a\,_-^b\,c\,_-^a\,d\,_-^c\,b\,_-^d),\end{eqnarray}
in which case $\mu$ becomes periodic.  
\end{proof}

\begin{question}\label{preg}
Considering that Y-maps $M(G)$ as in Definition~\ref{25} have each of their belts being a Y-belt, that is: with Y-edges (i.e., those not having their end-vertices in 4-belts) as opposite edges, are there any other constraints on Y-maps $M(G)$ that impede spray propagation reaching the formation of an ETGC for $M(G)$?
\end{question}

\begin{corollary}\label{under} Under the conclusion of Theorem~\ref{tt}, if an ETGC is obtained, 
 then there are two mutually orthogonal ETGCs $\mu,\mu'$ on $M(G)$ over a common vertex coloring $\mu_V=\mu'_V$ of $V=V(G)$ with no common color of $\mu$ and $\mu'$ on any $e\in E=E(G)$. These two ETGCs $\mu$ and $\mu'$ guarantee the existence of an EGC on the prism $G\square K_2$.
In particular, 
the 3-cube graph $Q_3$ and all the prisms $C_{4j}\square K_2$ ($1<j\in\mathbb{Z}$) have Y-maps with ETGCs via color set $[3]_0$ in two mutually orthogonal ways. Moreover, any two resulting mutually orthogonal ETGCs guarantee an EGC in the 4-cube graph $Q_4=Q_3\square K_2$ and in each prism $(C_{4j}\square K_2)\square K_2$.
\end{corollary}

\begin{figure}[htp]
\hspace*{1.5cm}
\includegraphics[scale=1.26]{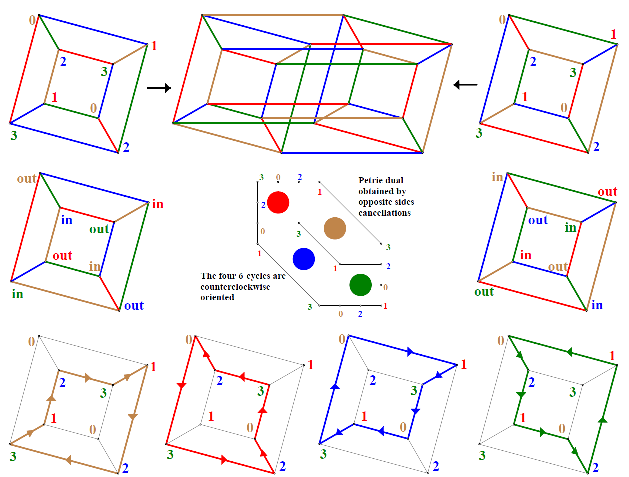}
\caption{Illustration for Corollary~\ref{under}. Colors as in (\ref{colors}).}
\label{schem}
\end{figure}

\begin{proof}
 The general situation for any $G$ under the assumptions works similarly as in Example~\ref{santo}. Namely, if $\mu_0$ and $\mu_1$ stand for the ETGCs on the left and right copies of a $G$ as in the top third of Figure~\ref{schem}, respectively, $\mu'$ stands for the resulting EGC on $G\square K_2$, $V(K_2)=\{w_0,w_1\}$ and $(u,v)\in E(G)$, then $\mu'((u,w_i),(v,w_i))=\mu_i(u,v)$, for $i=0,1$, 
while being $\mu_0(u)=\mu_1(u)$ the color of vertex $u$ in (either copy of) $G$, then $\mu'((u,w_0),(u,w_1))=\mu_0(u)=\mu_1(u)$.
\end{proof}

\begin{figure}[htp]
\hspace*{2cm}
\includegraphics[scale=0.9]{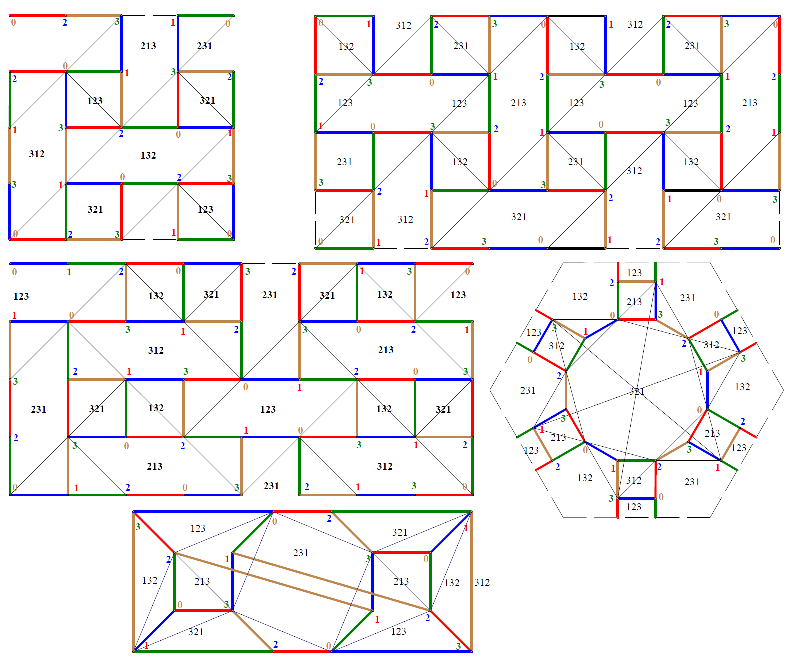}
\caption{Illustration for Example~\ref{is-bip}. Colors as in (\ref{colors}).}
\label{is-bipar}
\end{figure}

\begin{example}\label{santo}
The right side of  (\ref{oct})
shows two mutually orthogonal ETGCs on the standard map $M(Q_3)$ with a common vertex coloring represented in a $(4\times 1)$-xcutout. 
The top third of Figure~\ref{schem} represents on its left and right sides these two ETGCs, where colors
are given by numbers in $[3]_0$ on vertices and by colors according to (\ref{colors}) on edges. These colors are combined in the center of such top third of the figure yielding the claimed EGC on the cartesian product $G\square K_2$.
\end{example}

\subsection{Directed Petrie duals}\label{imanes}

\begin{theorem}\label{bip}
Let $M(G)$ be a map of a connected simple cubic graph of girth 4 with all its belt lengths divisible by 4 in an orientable surface $S$. Then, $G$ is bipartite.
\end{theorem}

\begin{proof} 
The square $G^2$ of $G$ has two connected components whose vertex sets constitute the two parts, say $V_1,V_2$, of the claimed bipartite graph $G$, so $V_1\cup V_2=V(G)$ and $V_1\cap V_2=\emptyset$. Both $V_1$ and $V_2$ are edge-disjoint unions of triangles. Each 4-belt of $M(G)$ contributes two non-adjacent vertices to each of $V_1,V_2$. Each other belt $B$ of $M(G)$ is a $4\ell$-belt, for some $\ell>1$, so $V_1,V_2$ contain alternate vertices of $V(B)$, yielding two independent vertex subsets. Using the connectedness and cubicity of $G$ and the orientability of $S$, a simple continuation algorithm allows to decide whether each vertex of $G$ is either in $V_1$ or $V_2$ for being at distance 2 for previously processed vertices. 
\end{proof}

\begin{example}\label{is-bip}
Figure~\ref{is-bipar} contains some toroidal cutouts of maps $M(G)$ as in the statement of Theorem~\ref{bip} that form part of Figures~\ref{semi}-\ref{noera}-\ref{sqsq}-\ref{20-gray} but modified by the addition of some edges of one of the two components of $G^2$ in thin trace in contrast to the thick trace for the colored edges of the graphs $G$. This can be done on any such cubic map $M(G)$ in orientable surfaces. However, if the surface is non-orientable, odd cycles may appear, as is the case of the graph obtained as the edge-disjoint union of an 8-cycle with opposite-vertex spokes. This is embeddable in the M\"obius band from the xcutout on the upper-left of Figure~\ref{semi} by identifying the border vertical edges with opposite orientation, which indeed creates 5-cycles.
\end{example}

\begin{definition}\label{petrie} Let $G$, $M(G)$, $X$, $v$, $e$, $f$ and $g$ be as in Definition~\ref{sabado}.
Initializing the spray by coloring $\{v,e,f,g\}$ in bijection with $[3]_0$ and continuing via items 1-2 over the rest of $M(G)$, if an ETGC $\mu$ is obtained,
then, $G$ is bipartite with vertex parts $V_L,V_R$ so that
each $e\in E(G)$ is the union of two arcs: $a_e^\rightarrow$ from $V_L$ to $V_R$ and $a_e^\leftarrow$ from $V_R$ to $V_L$.    
Let $G^\rightarrow$ and $G^\leftarrow$ be the digraphs with common vertex set $V(G)$ and arc sets $A^\rightarrow=\{a_e^\rightarrow;e\in E(G)\}$ and $A^\leftarrow=\{a_e^\leftarrow;e\in E(G)\}$, respectively. We set $\mu$ as ETGC $\mu^\rightarrow$ of $G^\rightarrow$, and  $\mu'$ as ETGC $\mu^\leftarrow$ of $G^\leftarrow$, meaning that the color of an arc $a^\rightarrow_e$ or $a^\leftarrow_e$ is a color, $\mu_E(e)$ or $\mu'_E(e)$, of its containing edge $e$. The union $G^\leftrightarrow=G^\rightarrow\cup G^\leftarrow$ of digraphs $G^\rightarrow$ and $G^\leftarrow$ with common vertex set $V(G)$ and disjoint arc sets $A^\rightarrow$ and $A^\leftarrow$ admits the efficient {\it directed total}, or {\it ditotal}, girth coloring (or {\it EDGC}) $\mu^\leftrightarrow$ obtained by adopting $\mu^\leftrightarrow_V=\mu_V$ over $V=V(G)$, $\mu^\leftrightarrow_{A^\rightarrow}=\mu^\rightarrow_{A^\rightarrow}$ over $A^\rightarrow$ and $\mu^\leftrightarrow_{A^\leftarrow}=\mu^\leftarrow_{A^\leftarrow}$ over $A^{\leftarrow}$. 
A {\it monochromatic circuit} of $(G^\leftrightarrow,\mu^\leftrightarrow)$ is a directed cycle formed by arcs of $A^\rightarrow\cup A^\leftarrow$ with a specific common color in $[3]_0$.
\end{definition}

\begin{theorem}
The EDGC $\mu^\leftrightarrow$ restricts to four bijections onto the color set $[3]_0$, namely departing from the sets formed by: {\bf(a)} each vertex and the arcs arriving to it; {\bf(b)} each vertex and the arcs departing from it; {\bf(c)} each vertex and its neighbors; {\bf(d)} each pair of oppositely oriented arcs of an edge and its end-vertices. 
\end{theorem}

\begin{proof}
Straightforward from Definition~\ref{petrie}.
\end{proof}

\begin{theorem}\label{jaramillo} 
The map over $(G^\leftrightarrow,\mu^\leftrightarrow)$ whose faces are determined by the belts formed by its monochromatic circuits is well defined.
\end{theorem}

\begin{proof}
We make each monochromatic circuit $C$ of $(M(G^\leftrightarrow)$ to be the border of a homeomorph of an open disk $D_C$ of $\mathbb{R}^2$. Then, the union of all monochromatic circuits $C$ and homeomorphs $D_C$ is the map claimed in the statement.  
\end{proof}

\begin{definition}
The {\it Petrie dual} $P(G^\leftrightarrow,\mu^\leftrightarrow)$ of $(G^\leftrightarrow,\mu^\leftrightarrow)$ is the well-defined map of Theorem~\ref{jaramillo}. 
\end{definition}

\begin{example}
We interpret the representations of $G=Q_3$ on the left and right of the middle third of Figure~\ref{schem} as copies of $(G^\rightarrow,\mu^\rightarrow)$ and $(G^\leftarrow,\mu^\leftarrow)$, respectively, where the vertices, instead of being referred with their color numbers in $[3]_0$, as in the top third of the figure, are referred by the prefixes``out" and ``in" for out-vertices (sources) and in-vertices (targets) of the composing arcs of $A^\rightarrow$ and $A^\leftarrow$. The union of these representations of $(G^\rightarrow,\mu^\rightarrow)$ and $(G^\leftarrow,\mu^\leftarrow)$ yields a representation of $(G^\leftrightarrow,\mu^\leftrightarrow)$. Observe that the arc set of $(G^\leftrightarrow,\mu^\leftrightarrow)$ partitions into four monochromatic 2-factors, in fact 6-cycles, thus yielding a 2-factor decomposition of the Petrie dual of $(G^\leftrightarrow,\mu^\leftrightarrow)$.  These monochromatic directed belts are represented in the lower third of Figure~\ref{schem} in four separate representations of $Q_3$, and altogether in the center of the middle third as a toroidal bicutout $X$ from which the Petrie dual is recovered by identification of opposite sides of the border of $X$; the four monochromatic Petrie belts here have their respective colors referred via solid colored circles in their apparent interiors.
\end{example}

\begin{theorem}\label{eliquis} Under the assumptions of Corollary~\ref{under}, the {\it Petrie dual} $P(G^\leftrightarrow,\mu^\leftrightarrow)$ is a map with an EDGC.
\end{theorem}

\begin{proof} 
We make each monochromatic circuit $C$ of $M(G^\leftrightarrow)$ to be the border of a homeomorph of an open disk $D_C$ of $\mathbb{R}^2$. Then, the union of all monochromatic circuits $C$ and homeomorphs $D_C$ is the map claimed in the statement admitting an EDGC as in Definition~\ref{petrie}.  
\end{proof}

Additional examples of directed Petrie duals with monochromatic circuits can be found in the Appendix, Section~\ref{appendix}.

\subsection{Periodic partitions into EDSs}

\begin{definition}\label{bicudo} 
Let us order as a list $\mathcal{L}$ the vertices of the belts of each Y-map $M(G)$ with the vertices of each belt presented cyclically and each new belt in  $\mathcal{L}$ appearing adjacent to some previous belt. 
We only consider Y-maps $M(G)$ that admit partitions into EDSs obtained by the continuation algorithm of denoting the vertices of $\mathcal{L}$ periodically via $[3]_0$ so that the vertex color cycle of each $\ell$-belt is of the form $(a,b,c,d,\ldots,a,b,c,d)=(a,b,c,d)^{\frac{\ell}{4}}$ with period $(a,b,c,d)$, where $\{a,b,c,d\}=[3]_0$. A Y-map that admits such a partition will be said to be a {\it $\Pi$-map}. Such a Y-map clearly admits ETGCs that are periodic.
A Y-graph $G$ having a $\Pi$-map $M(G)$ will be said to be a {\it $\Pi$-graph}. 
\end{definition}

\begin{example}
On the left side of (\ref{oct}), a bicutout of a toroidal map yields a case of a non-$\Pi$-graph, not even Y-graph (Example~\ref{contrera}). 
In the lower-left corner of Figure~\ref{abel}, a 6-cutout of a toroidal map $M(G)$ of a cubic graph $G$ of girth 4 with all belts of lengths divisible by 4 is shown that is not a $\Pi$-map, where $V(G)$ has colors as in  (\ref{colors}).
This 6-cutout is obtained as a modification of the 6-cutout on the left of Figure~\ref{20-gray} that is used in Examples~\ref{contra1}-\ref{contra2}. Such 6-cutout is followed to its right by the bicutout on the right of Figure~\ref{kk} (Example~\ref{klein}) which is also shown not to be a $\Pi$ map. 
\end{example}

\begin{question}\label{mapy} 
Can we consider in Definition~\ref{bicudo} a larger class of maps $M(G)$ than that of Y-maps, that is: not necessarily with pairs $(e_h,e_{h+\ell})$ of opposite edges in $2\ell$-belts of $M(G)$ with associated cycles $(e_1,e_2,\ldots,e_{2\ell-1},e_{2\ell})$ of subsequently adjacent edges, where $1\le h\le\ell$? 
\end{question}

\begin{example}\label{contra1} Section~\ref{p} presents a family of $\Pi$-graphs $G$ having both $\Pi$-maps and non-$\Pi$-maps. For example,
the $\Pi$-map $M(G)$ represented by the bicutout $\Upsilon$ on the right of Figure~\ref{20-gray}, where $G$ is a toroidal cubic of girth 4, has the set $\mathcal{B}$ (see Theorem~\ref{tt}) formed by the following clusters: the cluster $C_1$ given by the leftmost entire 4-belt $B_1$, the 4-belt $B_2$ divided in two by the vertical sides of $\Upsilon$, and the rightmost entire 4-belt $B_3$;
the cluster $C_2$, given by the uppermost entire 4-belt $B_4$, the 4-belt $B_5$ divided in two by the horizontal sides of $\Upsilon$, and the lowermost entire 4-belt $B_6$; and the central cluster $C_3$ with 4-belts $B_7,B_8,B_9$, counted downward. The O-arc $O_1$ in the 16-cycle of (\ref{nelly2}) departs from the leftmost vertex of $B_1$ towards the left, arrives to the rightmost vertex of $B_3$ (from the right) covering its top edge and vertex and reaching the bottom vertex of $B_4$ and its subsequent hazel edge. The O-arc $O_2$ departs upward from the leftmost vertex of $B_4$, arrives at the lowest vertex of $B_6$ (from the bottom side of $\Upsilon$) going to the right to the lowest vertex and edge of $B_9$. The first copy of $O_3$ goes from the rightmost vertex of $B_9$ to the right arriving from the left to the leftmost vertex and edge of $B_1$. The second copy of $O_3$ departs from the bottom vertex of $B_1$
going downward to the bottom border of $\Upsilon$ and arriving from its top border to the top vertex and red edge of $B_4$. Finally, $O_4^{-1}$ departs from the leftmost vertex of $B_4$ to the left, passing over the top edge of $B_7$ and arriving at the rightmost vertex and subsequent top edge of $B_1$, thus completing the 16-belt. Now, $O_2$ and $O_4$ have $x\ne d$. So, $M(G)$ is not a $\Pi$-map. However, $G$ is a $\Pi$-graph, see the left side of Figure~\ref{20-gray}.
The16-belt is undirected, which is also the case of the two 8-belts of $M(G)$, see (\ref{nelly3}).
\end{example}

\begin{definition}\label{opto} Let  $\mu$ be an ETGC of a $\Pi$-map $M(G)$. A belt of $M(G)$ is said to be {\it directed} (resp. {\it reversed}) if it has each color-0 edge $e$ in the clockwise (resp. counterclockwise) direction starting at a color-0 vertex $v$, meaning that there is distance 1 from $v$ to the nearest end-vertex of $e$.
If every belt of an ETGC $\mu$ is directed (resp.  reversed), we say that $\mu$ is a {\it DETGC}, (resp. {\it RETGC}). 
A belt that is neither directed nor reversed is said to be {\it undirected}. An ETGC of $M(G)$ is said to an {\it UETGC} if it contains an undirected belt.
\end{definition}

\begin{lemma}  Let  $\mu$ be an ETGC of a $\Pi$-map $M(G)$. 
Corresponding to the vertex coloring $\mu_V$ in $\mu$, there are two mutually orthogonal subjacent edge colorings $\mu^D$ and $\mu^R$ leading to two respective  mutually orthogonal ETGCs distinguishable as the DETGC $\mu_V^D=(\mu_V,\mu^D)$ and the RETGC $\mu_V^R=(\mu_V,\mu^R)$. Then, either $\mu=\mu_V^D$, or $\mu=\mu_V^R$. 
\end{lemma}

\begin{proof}
Let $v\in V(G)$ have $\mu_V(v)=0$ and let $B$ be a belt of the $\Pi$-map $M(G)$ containing $v$. Then, either the nearest color 0 edge $e$ in $B$ is in the clockwise direction with respect to $v$, in which case $B$ is directed, or $e$ is in the counterclockwise direction with respect to $v$, in which case $B$ is reversed. Since $M(G)$ is a $\Pi$-map, all belts of $M(G)$ are either directed or reversed, so either $\mu$ is a DETGC $\mu=\mu_V^D$ or $\mu$ is an RETGC $\mu=\mu_V^R$.
\end{proof}

\begin{example}\label{contra2} The ETGC represented in the upper-right $(4\times 1)$-xcutout of (\ref{oct}) has,
in every belt, each color-0 edge $e$ nearer in the counterclockwise direction to a color-0 vertex $v$, meaning that there is distance 1 from $v$ to the nearest end-vertex of $e$. In such a case, we say that the ETGC is a DETGC. An RETGC is the case of the ETGC represented in the lower-right $(4\times 1)$-xcutout of (\ref{oct}), namely with each color-0 edge $e$ nearer in the clockwise direction to a color-0 vertex $v$ in each belt. 
However, in the 16-belt on the right cutout $\Upsilon$ in Figure~\ref{20-gray} covering the 4-corners vertex of $\Upsilon$, with color cycle decomposed in (\ref{nelly2}) into O-arcs, the lowest color-0 vertex and color-0 edge are at distance more than one from any other color-0 edge or color-0 vertex, respectively. In such a case, the ETGC is a UETGC. This is also the case 
of the 8-belt (\ref{nelly1}) in the bicutout on the right of (\ref{te}).  
\end{example} 

\begin{proposition}
DETGCs (resp. RETGCs) only happen as on the left (resp. right) of (\ref{reverse}) with $x=d$.
On the other hand, in a UETGC, (\ref{reverse}) occurs with $x\ne d$ at least once, as it contains undirected belts, not complying with the cyclic disposition of (\ref{sche}) or its reverse. Thus, both DETGCs and RETGCs are periodic ETGCs, while UETGCs are aperiodic ETGCs.
\end{proposition} 

\begin{proof}
In the case of DETGCs (resp.  RETGCs), any instance of the left (resp. right) side of (\ref{reverse}) must have $x=d$ because every belt is cyclic as in (\ref{sche}) (resp. its reverse). This fails at least once in the case of UETGCs.
\end{proof}

\subsection{Face $S_3$-colorings}\label{face}

Let $G$ be a connected simple cubic graph of girth 4 and genus equal to that of an orientable surface $S$. Let $M(G)$ be a $\Pi$-map of $G$ in $S$. Assume that the faces of $M(G)$ have their $\ell$-belts with $\ell\equiv 0 \mod 4$. Given an ETGC $\mu':V(G)\cup E(G)\rightarrow[3]_0$, let the faces of $M(G)$ have $\mu'$-colored $\ell$-belts resulting from periodic concatenations of one of the two path patterns shown in  (\ref{reverse}) always with $x=d$. Then, a 3-permutation $\mu(B)$ is associated to each belt $B$ of $M(G)$, as follows.
Let $0,a,b,c$ in that order be the colors of $[3]_0$ assigned by $\mu'$ to any set of four consecutive vertices starting with 0 and in the clockwise direction of $B$, where $a,b,c$ are pairwise different and distinct from $0$. We assign to $B$ the 3-permutation $\mu(B)=abc$ expressed as $\mu(B)=abc$. This yields a face $S_3$-coloring of $M(G)$. 
 
\begin{theorem}\label{fifo} A periodic ETGC $\mu'$ of a $\Pi$-map $M(G)$ obtained as in Theorem~\ref{tt} has an associated face $S_3$-coloring $\mu$ such that no two adjacent faces of $M(G)$ are assigned by $\mu$ the same color in $S_3$.
\end{theorem}

\begin{proof}

Let $B,B'$ be belts of $M(G)$ sharing an edge $(u,v)\in E(G)$. Let $\mu(B)=abc$. Then:

if $\mu\{u,v\}=\{0,a\}$, then $\mu(B')=bca$;\;\; if $\mu\{u,v\}=\{a,b\}$, then $\mu(B')=bac$;

if $\mu\{u,v\}=\{b,c\}$, then $\mu(B')=acb$;\;\; if $\mu\{u,v\}=\{c,0\}$, then $\mu(B')=cab$. 

\noindent Thus, no two adjacent faces of $M(G)$  are assigned by $\mu$ the same color in $S_3$.
\end{proof}

\begin{definition}\label{pipi}
A face $S_3$-coloring $\mu$ of $M(G)$ obtained as in Theorem~\ref{fifo} will be said to be a {\it PFC} for {\it permutation face coloring}.
\end{definition}

\begin{theorem}\label{lapsus} Let $G$ be a connected simple cubic graph of girth 4 and let $M(G)$ be a $\Pi$-map of $G$ in a surface $S$ of genus $g(G)$. Then,
there is a bijection between the DETGCs (resp. RETGC) $\mu'$ of $M(G)$ and the PFCs $\mu$ of its faces.
\end{theorem}

\begin{proof} We consider only the DETGC case, as the RETGC case is similar.
From a PFC $\mu$ of $M(G)$, the vertex coloring of a DETGC $\mu'$ is obtained as follows.
If the edge $e$ separating two faces $B$ and $B'$ of $M(G)$ has its end-vertices assigned colors $i,j\in\{1,2,3\}$, then the transposition $(i\;j)$ relates $\mu(B)$ and $\mu(B')$, meaning that $(i\;j)$ exchanges 
$\mu(B)$ and $\mu(B')$.

Otherwise, one end-vertex of $e$ has assigned color 0. Say $i\in\{1,2,3\}$ is the color of the other end-vertex of $e$, in which case $\mu(B)$ and $\mu(B')$ are of the forms $ijk$ and $jki$, respectively,  or vice versa and are related by a 3-rotation, namely either $(1\;2\;3)$ or $(3\;2\;1)$, taking $\mu(B)$ and $\mu(B')$ respectively onto $\mu(B')$ and $\mu(B)$. 

The resulting vertex coloring satisfies the efficient domination condition. In fact, consider three faces of respective belts $B_1,B_2,B_3$ sharing a vertex $v$. We consider two cases.

\begin{enumerate}\item
If $\mu(\{B_1,B_2,B_3\})=\{123,231,312\}$, then $\mu(v)=0$. Then, by removing from $M(G)$ the closed neighborhoods of all such vertices $v$, we get a 2-factor $\mathcal{F}_0$ whose vertices have subsequent colors forming color cycles $(123\cdots 123\cdots 123)$, or the reverse. Then, the vertices $v=v_0$ with color 0 are at distance 1 from $\mathcal{F}_0$ via edges reaching $\mathcal{F}_0$ at different points. Moreover, those vertices $v$ form an EDS.
\item If $\mu(\{B_1,B_2,B_3\})=\{abc,bac,acb\}$, where $b$ is the only member of $\{a,b,c\}$ occupying the first, second and third positions of the 3-permutations, then $\mu(v)=b$. Then, by removing from $M(G)$ the closed neighborhoods of all such vertices $v=v_b$, we get a 2-factor $\mathcal{F}_b$ whose vertices have subsequent colors forming color cycles $(0ca\cdots 0ca\cdots 0ca)$, or their reverse. Then, the vertices $v=v_b$ with color $b$ are at distance 1 from $\mathcal{F}_b$ via edges reaching $\mathcal{F}_b$ at different points. Moreover, those vertices $v_b$ form an EDS, where $b$ takes the three values in $\{1,2,3\}$.
\end{enumerate}

Let $\psi$ be the assignment that sends the DETGCs $\mu'$ of $M(G)$ to their corresponding PFCs $\mu$ as defined above. Let $\mu_1,\mu_2$ be two different DETGCs
of $M(G)$ such that $\psi(\mu_1)=\psi(\mu_2)$. Given two belts $B,B'$ sharing an edge $e=(u,v)$ in $M(G)$, we have the following possibilities, where $i=1,2$:
\begin{enumerate}
\item $(\psi(\mu_i(B))=abc\mbox{ and }\psi(\mu_i(B'))=acb)\Rightarrow \mu_i(e)=(b\;c)\Rightarrow\mu'_i\{u,v\}=\{b,c\}$;
\item $(\psi(\mu_i(B))=abc\mbox{ and }\psi(\mu_i(B'))=bac)\Rightarrow \mu_i(e)=(a\;b)\Rightarrow\mu'_i\{u,v\}=\{a,b\}$;
\item $(\psi(\mu_i(B))=abc\mbox{ and }\psi(\mu_i(B'))=cab)\Rightarrow \mu_i(e)=(a\;c\;b)\Rightarrow\mu'_i\{u,v\}=\{0,c\}$;
\item $(\psi(\mu_i(B))=abc\mbox{ and }\psi(\mu_i(B'))=bca)\Rightarrow \mu_i(e)=(a\;b\;c)\Rightarrow\mu'_i\{u,v\}=\{0,a\}$.
\end{enumerate}
If $v\in V(G)$ is the common vertex of the belts $B_1,B_2,B_3$ and
\begin{enumerate}
\item $\mu'_i(\{B_1, B_2,B_3\})=\{abc,bca,cab\}$, then $\mu_i(v)=0$, where $i=1,2$;
\item $\mu'_i(\{B_1, B_2,B_3\})=\{abc,bac,bca\}$, then $\mu_i(v)=a$, where $i=1,2$.
\end{enumerate}
This proves that $\psi$ is injective. Now, the PFCs obtained from the DETGCs of $M(G))$ have the 3-permutations of any two adjacent belts related by a transposition or a rotation of $\{1,2,3\}$ which in turn allows to recover the vertex colors of $\mu'$. Thus, they form the image of the collection of DETGCs of $M(G)$, showing that $\psi$ is surjective over that image. 

Once the vertex colors of $G$ in $\mu'$ are recovered from $\mu$, Corollary~\ref{under} and the additional assumption that $x=d$ happens in all occurrences of (\ref{reverse}),  
shows that there are just two mutually orthogonal instances to complete $\mu'$ over $E(G)$,
one corresponding to a DETGC and the other to an RETGC. In our case, it is the DETGC version.
\end{proof}

\begin{figure}[htp]
\includegraphics[scale=0.63]{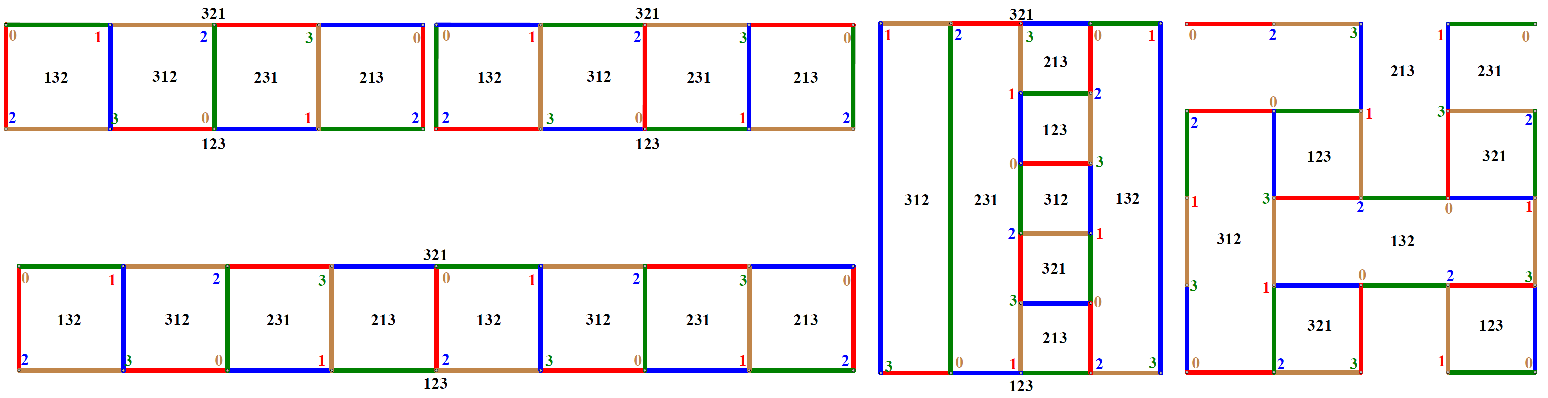}
\caption{ETGCs \& PFCs associated with (\ref{oct}) and Examples~\ref{repet}, \ref{ex3} and~\ref{truncated}. Colors as in (\ref{colors}).}
\label{semi}
\end{figure}

\begin{example}\label{ne}
In the upper-left of Figure~\ref{semi}, cutouts of the mutually orthogonal ETGCs on the right half of (\ref{oct}) are shown with a PFC $\mu(M(Q_3))$ from the six faces of the standard map $M(Q_3)$ onto the six 3-permutations 123, 132, 213, 231, 312 and 321, presented from left to right as 132, 312, 231 and 213 with top and bottom (horizontally presented) 4-belts assigned 3-permutations 321 and 123, respectively. Definition~\ref{pipi} yields that the left copy of $Q_3$ holds an RETGC, while the right copy of $Q_3$ holds a DETGC.  
\end{example}

\begin{remark}\label{lapsi}
From Theorem~\ref{lapsus} and its proof, we note that $E(G)$ has an {\it auxiliary coloring} with color set formed by the 3 transpositions and 2 rotations of $\{1,2,3\}$. This takes to an auxiliary coloring $\mu^*$ of the dual $\Pi$-map $M^*(G)$ of $M(G)$. Such dual $\Pi$-map is a triangulation of $S$ that has its vertex set $V^*(M(G))$ being the set $F(M(G))$ of faces of $M(G)$, with an edge $e^*=(B_1,B_2)\in E^*(M(G))$ for each edge $e\in E(G)$ separating faces of belts $B_1,B_2$ of $M(G)$. A vertex $v\in V(G)$ determines a triangular face $F^*(v)$ of $M^*(G)$ whose vertices $B_1^*,B_2^*,B_3^*$ are the faces of the respective belts $B_1,B_2,B_3$ containing $v$. The sides of such $F^*(v)$ are the edges $e_1^*,e_2^*,e_3^*$ dual to the corresponding edges $e_1,e_2,e_3$ that are incident to $v$. The degree of a vertex $B^*$ of $M^*(G)$ is a multiple of 4, namely the length $\ell$ of the corresponding belt $B$ of a face of $M(G)$. Now, $\mu^*$ assigns 3-permutations to the members of its vertex set $V^*(M(G))=F(M(G))$ and also to the dual edges $e^*$ of the edges $e\in E(G)$, as conceived above for the faces and edges of $M(G)$, where in the case of edges, they act as transpositions and rotations of $\{1,2,3\}$. Each face $v^*$ of $M^*(G)$ gets a color in $[3]_0$ from its surrounding coloring information as follows. If $bac, abc, acb$ (resp. $abc,bca,cab$) are the respective 3-permutations assigned by $\mu^*$ to the vertices $B^*_1,B^*_2,B^*_3$ of $v^*$, where $\{a,b,c\}=\{1,2,3\}$, then $\mu^*_G(v^*)=a$ (resp.  $\mu^*_G(v^*)=0$.
More specifically, the 3-permutations of the two triples here can be completed into auxiliary color cycles $(bac\;(a\;b)\;abc\;(b\;c)\;acb\;(a\;c\;b))$, (resp. $(abc\;(a\;b\;c)\;bca\;(b\;c\;a)\;cab\;(c\;a\;b))$. In particular, notice that if $G$ is not 3-edge connected, then $M^*(G)$ has pairs of parallel edges than in $S$ have corresponding non-contractible circles as their topological images in case the surface $S$ has genus larger than 0.
\end{remark}

\begin{example}\label{trunc0}
The {\it truncated square tiling} of $\mathbb{R}^2$ (see \cite{Branko}) restricts to the bicutout $X$ of the $\Pi$-map $M(G)$ of the toroidal vertex-transitive 16-vertex cubic graph $G$ of girth 4 shown on the left of  (\ref{tess}) with an  RETGC $\mu'$ indicated, where the (uniformly) truncated squares appear as 8-belts. 
The right of Figure~\ref{semi}
contains a bicutout representing this $M(G)$ and showing a PFC $\mu$ corresponding to the said $\mu'$.
The dual coloring $\mu^*$ of such $\mu$ on the dual $\Pi$-map $M^*(G)$ of $M(G)$ is represented as a colored graph on the left of Figure~\ref{ahiva}, where the dual vertices are indicated as circles containing the six corresponding 3-permutations. Each edge of $M^*(G)$ has the same color as its dual edge in $M(G)$. Each 3-colored triangle is the dual of a corresponding vertex of $G$ whose $\mu'$-color is the missing color in $[3]_0$ from the colors of its 3 shown edges. A bicutout realization of $M^*(G)$ can be traced over the disposition on the said left of Figure~\ref{ahiva} by placing a dual vertex on each face and traversing edges separating pairs of adjacent faces by means of lines representing the dual edges.
\end{example}

\subsection{Plan of the remaining sections}

In the following sections, we present six constructive tools to generate ETGCs of maps of finite connected simple graphs of girth 4 that together with the spray tool of the present section are conjectured to be sufficient in order to produce all existing ETGCs, see Conjecture~\ref{con1}. This is contained in a final section with concluding remarks, where other conjectures and questions are also presented. 

\section{Extensions}\label{extensions}

 \begin{definition}\label{pe} An {\it extension} of a cutout $X$ of a planar or toroidal $\Pi$-map of a finite simple cubic graph $G$ of girth 4 is a cutout or bicutout formed by the concatenation, or union in parallel, of copies $X_1,X_2,\ldots,X_n$ of $X$ ($1<n\in\mathbb{Z}$) that are cutouts of corresponding copies $M(G_1),M(G_2),\ldots,M(G_n)$ of $M(G)$, stacked horizontally (resp. vertically) along the common linear vertical (resp. horizontal) sides between each $X_i$ and subsequent $X_{i+1}$ ($1\le i<n$), namely by identifying the right (resp. top) side $L$ of $X_i$ with the left (resp. bottom) side $L'$ of $X_{i+1}$, where vertices and edges of $G_i$ in $L$ are glued with the corresponding vertices and edges of $G_{i+1}$ in $L'$. 
 \end{definition}

\begin{figure}[htp]
\includegraphics[scale=0.57]{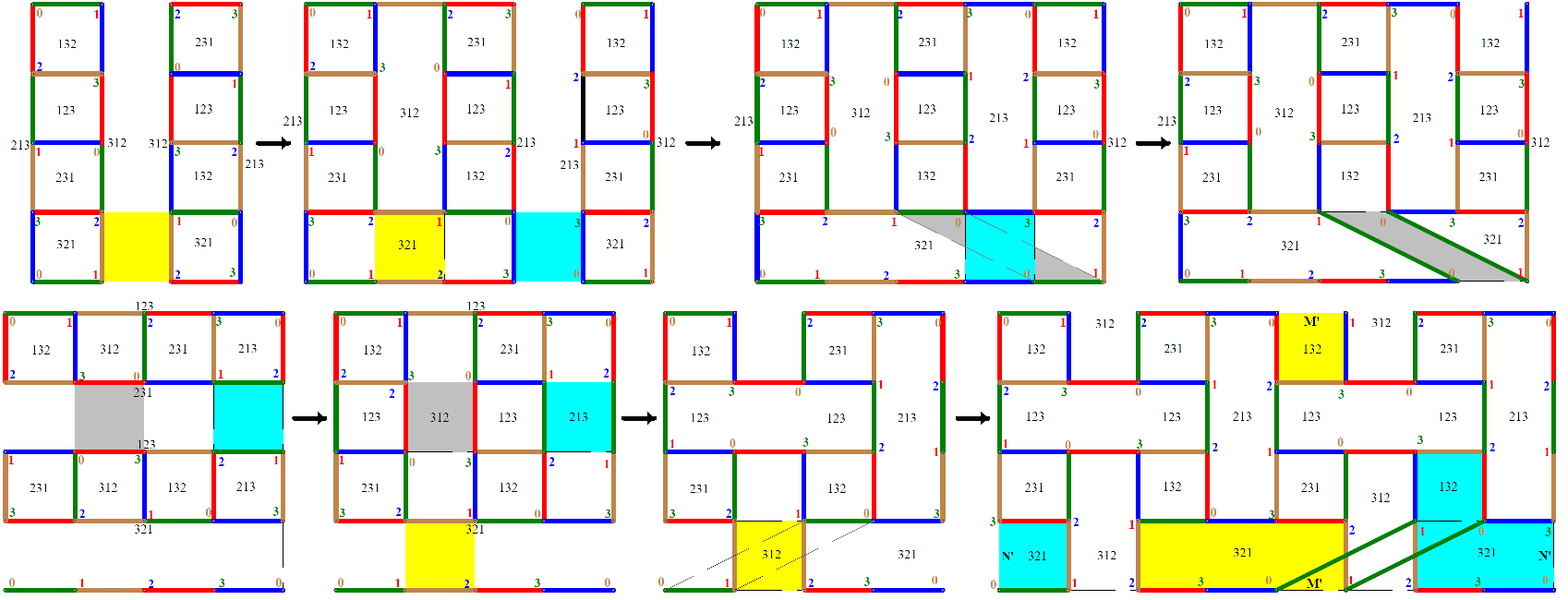}
\caption{The two sequences of exchanges in Example~\ref{tora}. Colors as in (\ref{colors}).}
\label{noera}
\end{figure}

\begin{theorem}\label{predni}
The successive identification of a finite number of copies of xcutouts (resp. ycutouts) or bicutouts of planar or toroidal $\Pi$-maps as in Definition~\ref{pe} yields a cutout or bicutout $\bar{X}$ of a planar or toroidal $\Pi$-map $M(\bar{G})$ of a connected simple cubic graph $\bar{G}$ of girth 4.
\end{theorem}

\begin{proof} Assume that $X$ is an xcutout. (The case of ycutout can be similarly tackled).
If $M(G)$ is planar, then the top and bottom faces of $M(G)$ are left out of $X$ as their belts are in fact embedded in the top and bottom sides of $X$, respectively. The other $|F(M(G))-2$ faces of $M(G)$ are reproduced in the horizontal concatenation of the copies $X_1,X_2,\ldots,X_n$ of $X$ ($1<n\in\mathbb{Z}$). Then,
$\chi(M(\bar{G}))=2+\sum_{i=1}^n(|V(G_i)|-|E(G_i)|+|F(M(G_i))|-2)=2+0=2$, so $g(\bar{G})=0$ and $\bar{G}$ is planar. If $G$ is toroidal, then $\chi(M(G))=0$ implies that 
$\chi(M(\bar{G}))=\sum_{i=1}^n(|V(G_i)|-|E(G_i)|+|F(M(G_i))|)=0$, so $g(\bar{G})=1$ and $\bar{G}$ is toroidal.
\end{proof}

\begin{question}\label{mapz}
Can the extensions of planar and toroidal cutouts and bicutouts as in Definition~\ref {pe} be generalized to $n$-cutouts based on $2n$-zonogons, wnere $n>1$? How? Can Theorem~\ref{predni} be extended by way of one such generalization? 
\end{question}

\begin{example}\label{repet}
If we glue successively a finite number of copies of, say, the leftmost $(4\times 1)$-bicutout in (\ref{oct}) and identify in parallel by an extension the leftmost and rightmost vertical edges of the resulting graph, a prism $C_{4j}\square K_2$ and an ETGC in it are obtained, as shown to the right of indication $(^{\times 2}_\rightarrow)$
in (\ref{oct2}), for $j=2$.

\begin{eqnarray}\label{oct2}\begin{array}{lll}
0\hspace*{2.2mm} _-^3\hspace*{2.2mm} 1\hspace*{2.2mm} _-^0\hspace*{2.2mm} 2\hspace*{2.2mm} _-^1\hspace*{2.2mm} 3\hspace*{2.2mm} _-^2\hspace*{2.2mm} 0 &&
0\hspace*{2.2mm} _-^3\hspace*{2.2mm} 1\hspace*{2.2mm} _-^0\hspace*{2.2mm} 2\hspace*{2.2mm} _-^1\hspace*{2.2mm} 3\hspace*{2.2mm} _-^2\hspace*{2.2mm}  0\hspace*{2.2mm} _-^3\hspace*{2.2mm} 1\hspace*{2.2mm} _-^0\hspace*{2.2mm} 2\hspace*{2.2mm} _-^1\hspace*{2.2mm} 3\hspace*{2.2mm} _-^2\hspace*{2.2mm} 0 \\
\!_1|\hspace*{6mm}_2|\hspace*{6mm}_3|\hspace*{6mm}_0|\hspace*{6mm}_1|&(_\rightarrow^{\times 2})&
\!_1|\hspace*{6mm}_2|\hspace*{6mm}_3|\hspace*{6mm}_0|\hspace*{6mm}_1|\hspace*{6mm}_2|\hspace*{6mm}_3|\hspace*{6mm}_0|\hspace*{6mm}_1|\\
2\hspace*{2.2mm} _0^-\hspace*{2.2mm}3\hspace*{2.2mm} _1^-\hspace*{2.2mm} 0\hspace*{2.2mm} _2^-\hspace*{2.2mm} 1\hspace*{2.2mm} _3^-\hspace*{2.2mm} 2 &&
2\hspace*{2.2mm} ^-_0\hspace*{2.2mm}3\hspace*{2.2mm} ^-_1\hspace*{2.2mm} 0\hspace*{2.2mm} ^-_2\hspace*{2.2mm} 1\hspace*{2.2mm} ^-_3\hspace*{2.2mm} 2\hspace*{2.2mm} ^-_0\hspace*{2.2mm}3\hspace*{2.2mm} ^-_1\hspace*{2.2mm} 0\hspace*{2.2mm} ^-_2\hspace*{2.2mm} 1\hspace*{2.2mm} ^-_3\hspace*{2.2mm} 2 \\
\end{array}\end{eqnarray}

In the lower-left of Figure~\ref{semi}, a PFC in an $(8\times 1)$-xcutout of the prism $C_8\square K_2$ following the pattern of the top left half of the figure, with top and bottom 8-belts colored with 321 and 123, respectively corresponds with the RETGC on the right of  (\ref{oct2}).
\end{example}

\section{Unfoldings}

\begin{definition}\label{unfold}
Given a cutout $\Phi$ of a $\Pi$-map $M(G)$ of a connected simple planar cubic graph $G$ of girth 4, an ({\it accordion}) {\it unfolding} of $\Phi$ is a cutout $\Phi'$ of a $\Pi$-map $M(G')$ of a connected planar cubic graph $G'$ obtained by the replacement of 4-belts of the form $K_2\square K_2=P_2\square P_2$ by copies of $K_2\square P_{2\ell}=P_2\square P_{2\ell}$, where $1<\ell\in\mathbb{Z}$ and $P_{2\ell}$ is a path of length $2\ell-1$.
\end{definition}

\begin{theorem}\label{starting}\cite[Theorem 16]{+1} Starting with the 4-belts of a cutout of a $\Pi$-map $M(G)$, where $G$ is a connected simple cubic graph of girth 4, successive unfoldings lead to an infinite family of planar maps $M(G')$ of connected cubic graphs $G'$ of girth 4. Each such $G'$ lacking $\ell$-belts, $\forall\ell\not\equiv 0 \mod 4$, admits ETGCs. Moreover, there are EGCs in the prisms $G'\square K_2$. 
\end{theorem}

\begin{proof} The proof of the statement is similar to that of~\cite[Theorem 16]{+1}.
\end{proof}   

\begin{figure}[htp]
\includegraphics[scale=0.585]{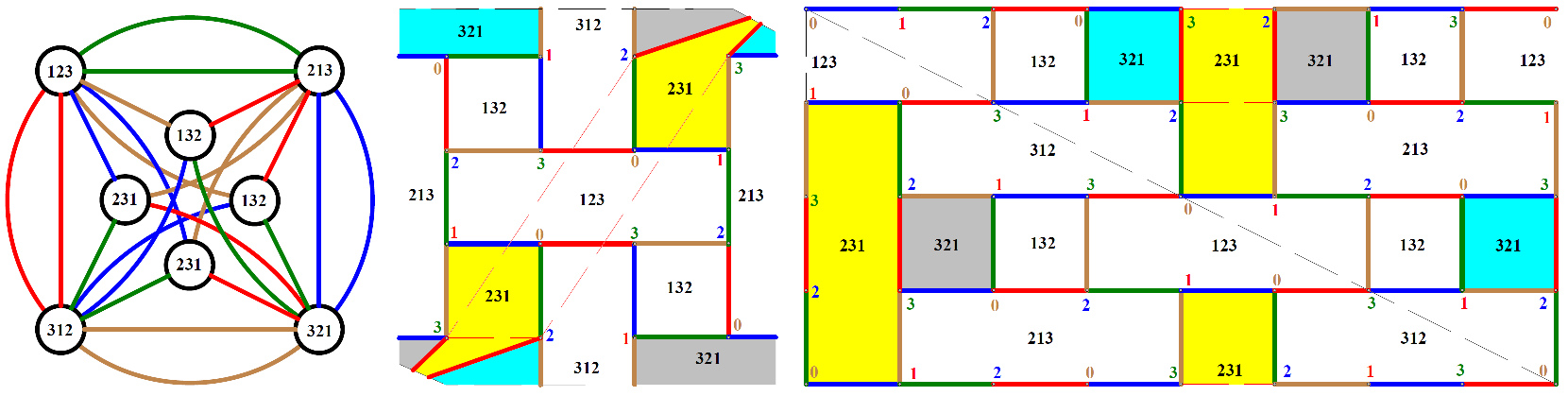}
\caption{A dual graph and exchanges not raising genus from 1 to 2. Colors as in (\ref{colors}).}
\label{ahiva}
\end{figure}

\begin{example}\label{ex3}
The leftmost $(4\times 1)$-xcutout in (\ref{oct3}) can be modified by replacing the 4-belt $H$ whose vertical edges are replaced by colons, by the transpose $(\cdots)^t$ of the copy of $P_6\square K_2$ to the right of union symbol $\cup$, with its leftmost and rightmost edges identified respectively to the corresponding horizontal edges of $H$.  On the center-right of Figure~\ref{semi}, a face-colored representation of an xcutout of the resulting cubic $\Pi$-map is given, with a 
DETGC and corresponding
PFC bearing un unfolding over the PFC $\mu(M(Q_3))$ of Example~\ref{ne}.
\begin{eqnarray}\label{oct3}\begin{array}{lllll}
1\hspace*{2.2mm} _-^0\hspace*{2.2mm} 2\hspace*{2.2mm} _-^1\hspace*{2.2mm} 3\hspace*{2.2mm} _-^2\hspace*{2.2mm}  0\hspace*{2.2mm} _-^3\hspace*{2.2mm} 1&&
3\hspace*{2.2mm} _-^0\hspace*{2.2mm} 1\hspace*{2.2mm} _-^2\hspace*{2.2mm} 0\hspace*{2.2mm} _-^3\hspace*{2.2mm} 2\hspace*{2.2mm} _-^1\hspace*{2.2mm}  3\hspace*{2.2mm} _-^0\hspace*{2.2mm} 1&\\
\!_2|\hspace*{6mm}_3|\hspace*{5mm}_0\!:\hspace*{5mm}_1\!:\hspace*{5mm}_2|&\cup\;(&
\!_2|\hspace*{6mm}_3|\hspace*{6mm}_1|\hspace*{6mm}_0|\hspace*{6mm}_2|\hspace*{6mm}_3|&)^t\\
3\hspace*{2.2mm} ^-_1\hspace*{2.2mm} 0\hspace*{2.2mm} ^-_2\hspace*{2.2mm} 1\hspace*{2.2mm} ^-_3\hspace*{2.2mm} 2\hspace*{2.2mm} ^-_0\hspace*{2.2mm}3&&
0\hspace*{2.2mm} ^-_1\hspace*{2.2mm}2\hspace*{2.2mm} ^-_0\hspace*{2.2mm} 3\hspace*{2.2mm} ^-_2\hspace*{2.2mm} 1\hspace*{2.2mm} ^-_3\hspace*{2.2mm} 0\hspace*{2.2mm} ^-_1\hspace*{2.2mm}2&\\
\end{array}\end{eqnarray}
Notice that this $\Pi$-map is equivalent to that of Example~\ref{repet}.
\end{example}

\begin{example}\label{truncated} In continuation to Example~\ref{trunc0}, we have that
Definition~\ref{unfold} is exemplified in (\ref{tess}) by the (rightmost) unfolding $X'$ of the (leftmost) bicutout $X$. 
\begin{eqnarray}\label{tess}\begin{array}{ll|ll}
0\hspace*{2.2mm}_-^1\hspace*{2.2mm}2\hspace*{2.2mm}_-^0\hspace*{2.2mm}3\hspace*{2.2mm}..\hspace*{2.2mm}1\hspace*{2.2mm}_-^3\hspace*{2.2mm}0&&&
0\hspace*{2.2mm}_-^1\hspace*{2.2mm}2\hspace*{2.2mm}_-^0\hspace*{2.2mm}3\hspace*{2.2mm}_-^2\hspace*{2.2mm}1\hspace*{2.2mm}_-^3\hspace*{2.2mm}
0\hspace*{2.2mm}..\hspace*{2.2mm}2\hspace*{2.2mm}_-^0\hspace*{2.2mm}3\hspace*{2.2mm}_-^2\hspace*{2.2mm}1\hspace*{2.2mm}_-^3\hspace*{2.2mm}0\\
:\hspace*{14.5mm}_2|\hspace*{6mm}_2|\hspace*{7mm}:&&&
:\hspace*{32mm}_1|\hspace*{6mm}_1|\hspace*{25mm}:\\
2\hspace*{2.2mm}_-^1\hspace*{2.2mm}0\hspace*{2.2mm}_-^3\hspace*{2.2mm}1\hspace*{7mm}3\hspace*{2.2mm}_-^0\hspace*{2.2mm}2&&&
2\hspace*{2.2mm}_-^1\hspace*{2.2mm}0\hspace*{2.2mm}_-^3\hspace*{2.2mm}1\hspace*{2.2mm}_-^2\hspace*{2.2mm}3\hspace*{2.2mm}_-^0\hspace*{2.2mm}
2\hspace*{7mm}0\hspace*{2.2mm}_-^3\hspace*{2.2mm}1\hspace*{2.2mm}_-^2\hspace*{2.2mm}3\hspace*{2.2mm}_-^0\hspace*{2.2mm}2\\
\!_3|\hspace*{6mm}_2|\hspace*{6mm}_0|\hspace*{6mm}_1|\hspace*{6mm}_3|&&&
\!_3|\hspace*{6mm}_2|\hspace*{6mm}_0|\hspace*{6mm}_1|\hspace*{7mm}
\!_3|\hspace*{6mm}_2|\hspace*{6mm}_0|\hspace*{6mm}_1|\hspace*{6mm}_3|\\
1\hspace*{7mm}3\hspace*{2.2mm}_-^1\hspace*{2.2mm}2\hspace*{2.2mm}_-^3\hspace*{2.2mm}0\hspace*{2.2mm}_-^2\hspace*{2.2mm}1&&&
1\hspace*{7mm}3\hspace*{2.2mm}_-^1\hspace*{2.2mm}2\hspace*{2.2mm}_-^3\hspace*{2.2mm}0\hspace*{2.2mm}_-^2\hspace*{2.2mm}
1\hspace*{2.2mm}_-^0\hspace*{2.2mm}3\hspace*{2.2mm}_-^1\hspace*{2.2mm}2\hspace*{2.2mm}_-^3\hspace*{2.2mm}0\hspace*{2.2mm}_-^2\hspace*{2.2mm}1\\
\!_0|\hspace*{6mm}_0|\hspace*{23.5mm}_0|&&&
\!_0|\hspace*{6mm}_0|\hspace*{60mm}_0|\\
3\hspace*{7mm}1\hspace*{2.2mm}_-^2\hspace*{2.2mm}0\hspace*{2.2mm}_-^3\hspace*{2.2mm}2\hspace*{2.2mm}_-^1\hspace*{2.2mm}3&&&
3\hspace*{7mm}1\hspace*{2.2mm}_-^2\hspace*{2.2mm}0\hspace*{2.2mm}_-^3\hspace*{2.2mm}2\hspace*{2.2mm}_-^1\hspace*{2.2mm}
3\hspace*{2.2mm}_-^0\hspace*{2.2mm}1\hspace*{2.2mm}_-^2\hspace*{2.2mm}0\hspace*{2.2mm}_-^3\hspace*{2.2mm}2\hspace*{2.2mm}_-^1\hspace*{2.2mm}3\\
\!_2|\hspace*{6mm}_3|\hspace*{6mm}_1|\hspace*{6mm}_0|\hspace*{6mm}_2|&&&
\!_2|\hspace*{6mm}_3|\hspace*{6mm}_1|\hspace*{6mm}_0|\hspace*{7mm}
\!_2|\hspace*{6mm}_3|\hspace*{6mm}_1|\hspace*{6mm}_0|\hspace*{6mm}_2|\\
0\hspace*{2.2mm} _-^1\hspace*{2.2mm}2\hspace*{2.2mm}_-^0\hspace*{2.2mm}3\hspace*{2.2mm}..\hspace*{2mm}1\hspace*{2.2mm}_-^3\hspace*{2.2mm}0&&&
0\hspace*{2.2mm} _-^1\hspace*{2.2mm}2\hspace*{2.2mm}_-^0\hspace*{2.2mm}3\hspace*{2.2mm}_-^2\hspace*{2mm}1\hspace*{2.2mm}_-^3\hspace*{2.2mm}
0\hspace*{2.5mm} ..\hspace*{2.5mm}2\hspace*{2.2mm}_-^0\hspace*{2.2mm}3\hspace*{2.2mm}_-^2\hspace*{2mm}1\hspace*{2.2mm}_-^3\hspace*{2.2mm}0\\
\end{array}\end{eqnarray}
\end{example}

\section{Exchanges}

\begin{figure}[htp]
\includegraphics[scale=0.58]{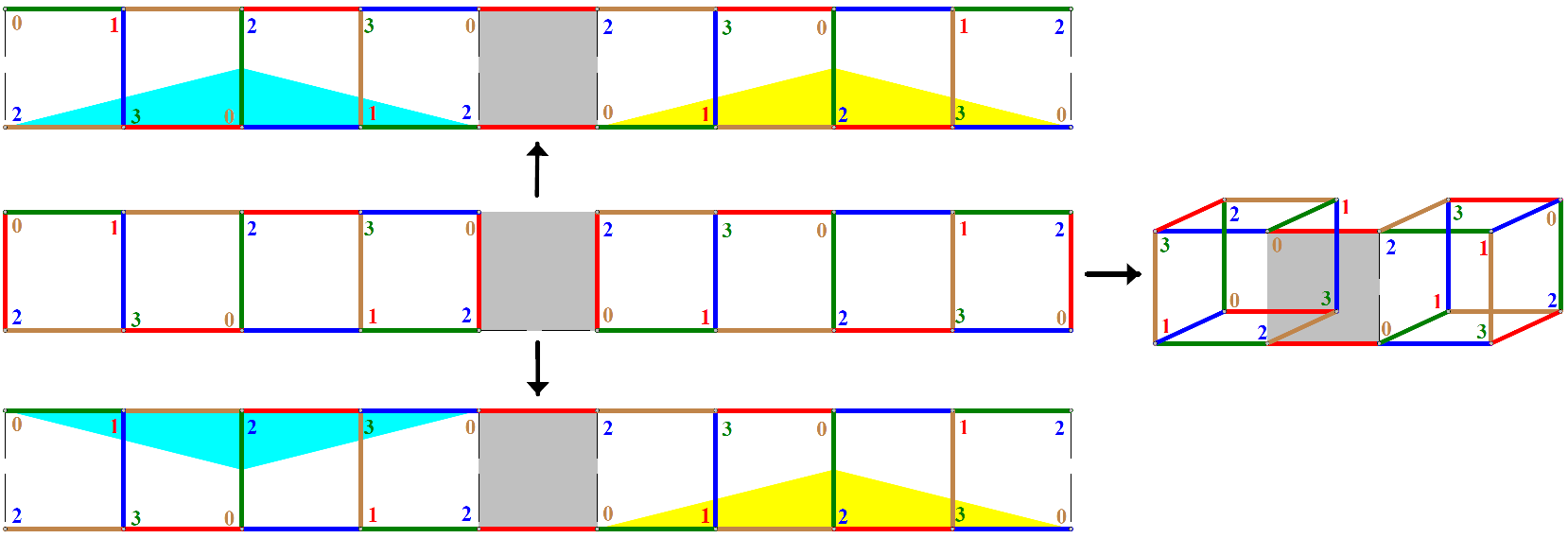}
\caption{Example~\ref{tor} of an exchange in the union of two 3-cubes. Colors as in (\ref{colors}).}
\label{torcuato}
\end{figure}

\begin{definition}\label{exchange} Let $G$ be a simple cubic graph of girth 4. Let $M(G)$ be a $\Pi$-map of $G$ and let $\mu:V(G)\cup E(G)\rightarrow[3]_0$
be either a DETGC or a RETGC. Let $C=(v_0,v_1,v_2,v_3)$ be a non-self-intersecting 4-cycle of $K_{|V(G)|}$ with opposite-edge pairs 
\begin{eqnarray}\label{c0c1}C_0=\{v_0v_1,v_2v_3\}\subset E(G)\;\;\mbox{ and }\;\;C_1=\{v_1v_2,v_3v_0\}\cap E(G)=\emptyset\end{eqnarray} such that
$\mu(v_0)=\mu(v_2)\neq \mu(v_1)=\mu(v_3)$ and $\mu(v_0v_1)=\mu(v_2v_3)$. 
If the graph $G'=G-C_0+C_1$ has girth 4 but no loops, parallel edges or triangles, then $G'$ is said to be obtained by {\it exchange} from $G$.
\end{definition}

\begin{theorem} Let $G'$ be obtained by exchange from $G$ as in the setting of Definition~\ref{exchange}.
Say that the edge $v_0v_1$ separates the faces of belts $B_0,B_1$ of $M(G)$ while the edge $v_2v_3$ separates the faces of belts $B_2,B_3$ of $M(G)$. Assume that $B_0,B_1,B_2,B_3$ are belts of four different faces of $M(G)$.
Assume that $B_i$ and $B_{i+2}$ are concatenations of a common O-arc $(a\,_-^b\,c\,_-^a\,d\,_-^c\,b\,_-^d)$, whose structure is exemplified in (\ref{sche}), for each one of  $i=0,1$. Then, $G'$ contains, instead of the belts $B_i$ ($i=0,1,2,3$) that were destroyed in passing from $G$ to $G'$, the two new cycles $$B_0\cup B_2\cup\{v_1v_2,v_0v_3\}\setminus\{v_0v_1,v_2v_3\}\;\;\mbox{ and }\;\;B_1\cup B_3\cup\{v_1v_2,v_0v_3\}\setminus\{v_0v_1,v_2v_3\}$$ interpretable as respective belts $B'_0,B'_1$ of a $\Pi$-map $M(G')$ on a resulting surface $S'$ whose genus $g(S')$ is one more than the genus $g(S)$ of $S$, namely $g(S')=g(S)+1$. 
\end{theorem}

\begin{proof}
The conclusion of the statement arises from the fact that the Euler characteristic of $G'$,  namely $$\chi(M(G'))=|V(G')|-|E(G')|+|F(M(G')|=|V(G)|-|E(G)|+|F(M(G)|-2=\chi(M(G)|-2$$ is diminished two units with respect to the Euler characteristic $\chi(G)=|V(G)|-|E(G)|+|F(M(G)|=2-2g(G)$ of $G$, as a result that the four belts $B_i$ ($i=0,1,2,3$) of $G$ are replaced in $G'$ by the new belts $B'_0,B'_1$, while all other belts of faces of $M(G)$ are preserved in $M(G')$. Then, $2-2g(G')=\chi(M(G'))=\chi(M(G))-2=2-2g(G)-2=-2g(G)$
implies $g(G')=g(G)+1$. The faces corresponding to the belts $B'_0,B'_1$ are separated by the new edges $v_1v_2,v_0v_3$ into the two halves of a resulting handle with two cyclic borders attachable to the holes produced by the removal of the interiors of the face unions $B_0\cup B_1\setminus(v_0v_1)$ and $B_2\cup B_3\setminus(v_2v_3)$. The raising of the genus in one unit is illustrated in Example~\ref{lastrem} and Figure~\ref{paquita}.  
\end{proof}

\begin{figure}[htp]
\includegraphics[scale=0.57]{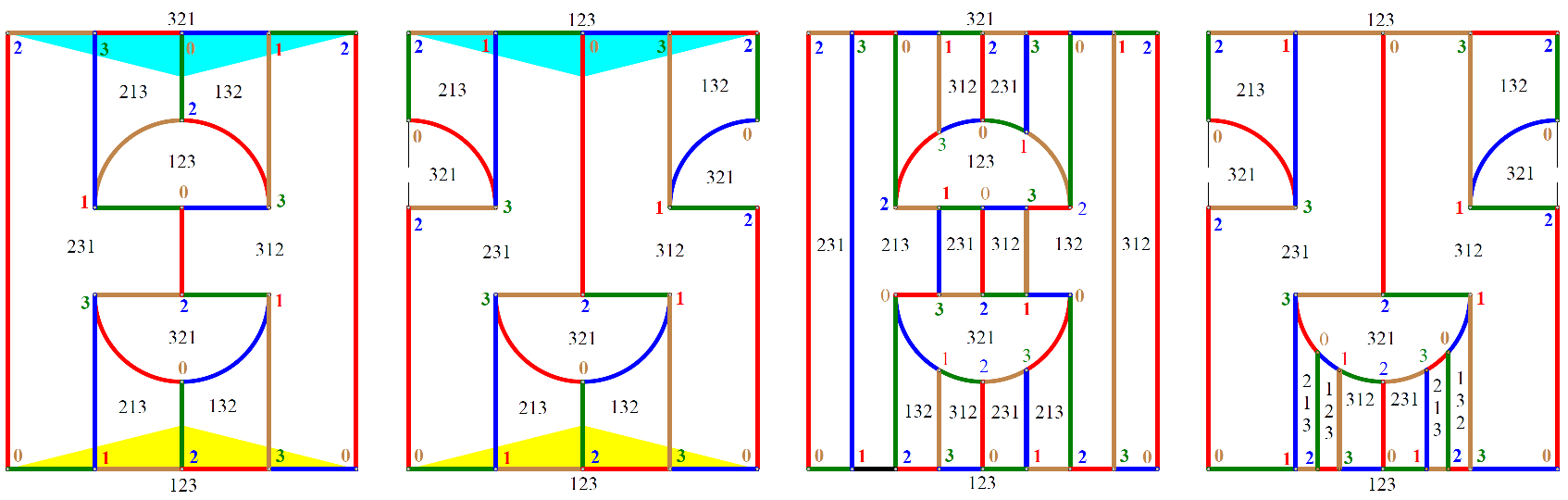}
\caption{Two variations of Figure~\ref{torcuato} and two related unfoldings. Colors as in (\ref{colors}).}
\label{bluyelow}
\end{figure} 

\begin{example}\label{tora}
 Figure~\ref{noera} consists of two horizontal sequences of exchanges, one on top of the other, both with PFCs as in Definition~\ref{pipi} equivalent to RETGCs as in Definition~\ref{opto}. 
 The top sequence starts on the left as the union of two $(1\times 4)$-ycutouts forming a $(3\times 4)$ycutout $X$ of a 2-component graph $G$ obtained by identifying the top and bottom horizontal borders of $X$. A 4-cycle $C$ with yellow interior allows an exchange of $G$ into a graph $G'$ obtained by top-bottom identification of a corresponding $(3\times 4)$-ycutout $X'$ shown to the immediate right of the leftmost ``$\rightarrow$" in the figure.  
The cutout $X'$ is also contained in the left-and-center of a $(5\times 4)$-ycutout $X''$ of the disjoint union of $G'$ and a 3-cube. In $X''$, a 4-cycle with light-blue interior is indicated for an exchange that leads to a $(5\times 4)$-ycutout $X'''$ of a graph $G''$ shown to the right of the second ``$\rightarrow$''. This sequence, presented so far also as in (\ref{octaedro}) (where exchange 4-cycles $C$ have $C_0$ given via segments and $C_1$ via colons or diaereses), extends to an infinite sequence of $((2k+1)\times 4)$-ycutouts by repeating the addition of a 3-cube to the right and a corresponding exchange. 

 \begin{eqnarray}\label{octaedro}\begin{array}{lllll}
1\hspace*{2.2mm} _-^2\hspace*{2.2mm} 0\hspace*{2.2mm} ..\hspace*{2.2mm} 2\hspace*{2.2mm} _-^1\hspace*{2.2mm} 3&&1\hspace*{2.2mm} _-^2\hspace*{2.2mm} 0\hspace*{2.2mm} _-^3\hspace*{2.2mm} 2\hspace*{2.2mm} _-^1\hspace*{2.2mm} 3\hspace*{2.2mm}..\hspace*{2.2mm}1\hspace*{2.2mm}_-^2\hspace{2.2mm}0&&
1\hspace*{2.2mm} _-^2\hspace*{2.2mm} 0\hspace*{2.2mm} _-^3\hspace*{2.2mm} 2\hspace*{2.2mm} _-^1\hspace*{2.2mm} 3\hspace*{2.2mm}_-^0\hspace*{2.2mm}1\hspace*{2.2mm}_{..}^2\hspace{2.2mm}0\\

\!_3|\hspace*{6mm}_1|\hspace*{6mm}_0|\hspace*{6mm}_2|&
&\!_3|\hspace*{6mm}_1|\hspace*{6mm}_0|\hspace*{6mm}_2|\hspace*{6mm}_3|\hspace*{6mm}_1|&&
\!_3|\hspace*{6mm}_1|\hspace*{6mm}_0|\hspace*{6mm}_2|\hspace*{6mm}_3|\hspace*{6mm}_1|\\

2\hspace*{2.2mm} ^-_0\hspace*{2.2mm}3\hspace*{7mm} 1\hspace*{2.2mm} ^-_3\hspace*{2.2mm} 0&&
2\hspace*{2.2mm} ^-_0\hspace*{2.2mm}3\hspace*{7mm} 1\hspace*{2.2mm} ^-_3\hspace*{2.2mm} 0\hspace{6.5mm}2\hspace*{2.2mm} ^-_0\hspace*{2.2mm}3&&
2\hspace*{2.2mm} ^-_0\hspace*{2.2mm}3\hspace*{7mm} 1\hspace*{2.2mm} ^-_3\hspace*{2.2mm} 0\hspace{6.5mm}2\hspace*{2.2mm} ^-_0\hspace*{2.2mm}3\\

\!_1|\hspace*{6mm}_2|\hspace*{6mm}_2|\hspace*{6mm}_1|&&
\!_1|\hspace*{6mm}_2|\hspace*{6mm}_2|\hspace*{6mm}_1|\hspace*{6.5mm}_1|\hspace*{6mm}_2|&&
\!_1|\hspace*{6mm}_2|\hspace*{6mm}_2|\hspace*{6mm}_1|\hspace*{6.5mm}_1|\hspace*{6mm}_2|\\

0\hspace*{2.2mm} ^-_3\hspace*{2.2mm}1\hspace*{7mm} 3\hspace*{2.2mm} ^-_0\hspace*{2.2mm} 2&\rightarrow&
0\hspace*{2.2mm} ^-_3\hspace*{2.2mm}1\hspace*{7mm} 3\hspace*{2.2mm} ^-_0\hspace*{2.2mm} 2\hspace*{7mm}0\hspace*{2.2mm} ^-_3\hspace*{2.2mm}1&\rightarrow&
0\hspace*{2.2mm} ^-_3\hspace*{2.2mm}1\hspace*{7mm} 3\hspace*{2.2mm} ^-_0\hspace*{2.2mm} 2\hspace*{7mm}0\hspace*{2.2mm} ^-_3\hspace*{2.2mm}1\\

\!_2|\hspace*{6mm}_0|\hspace*{6mm}_1|\hspace*{6mm}_3|&&
\!_2|\hspace*{6mm}_0|\hspace*{6mm}_1|\hspace*{6mm}_3|\hspace*{7mm}_2|\hspace*{6mm}_0|&&
\!_2|\hspace*{6mm}_0|\hspace*{6mm}_1|\hspace*{6mm}_3|\hspace*{7mm}_2|\hspace*{6mm}_0|\\

3\hspace*{2.2mm} ^-_1\hspace*{2.2mm}2\hspace*{2.2mm}..\hspace*{2.2mm} 0\hspace*{2.2mm} ^-_2\hspace*{2.2mm} 1&&
3\hspace*{2.2mm} ^-_1\hspace*{2.2mm}2\hspace*{2.2mm} ^-_3\hspace*{2.2mm} 0\hspace*{2.2mm} ^-_2\hspace*{2.2mm} 1\hspace*{2.2mm}..\hspace*{2.2mm}3\hspace*{2.2mm}^-_1\hspace*{2.2mm}2&&
3\hspace*{2.2mm} ^-_1\hspace*{2.2mm}2\hspace*{2.2mm} ^-_3\hspace*{2.2mm} 0\hspace*{2.2mm} ^{..}_2\hspace*{2.2mm} 1\hspace*{2.2mm}^-_0\hspace*{2.2mm}3\hspace*{2.2mm}^-_1\hspace*{2.2mm}2\\

\!_2|\hspace*{6mm}_0|\hspace*{6mm}_1|\hspace*{6mm}_3|&&
\!_2|\hspace*{6mm}:\hspace*{7mm}:\hspace*{6mm}_3|\hspace*{7mm}_1|\hspace*{6mm}_3|
&&
\!_2|\hspace*{25mm}:\hspace*{7mm}:\hspace*{5mm}_3|\\

1\hspace*{2.2mm} ^-_2\hspace*{2.2mm}0\hspace*{2.2mm}..\hspace*{2.2mm} 2\hspace*{2.2mm} ^-_1\hspace*{2.2mm} 3&&
\,1\hspace*{2.2mm} ^-_2\hspace*{2.2mm}0\hspace*{2.2mm} ^-_3\hspace*{2.2mm} 2\hspace*{2.2mm} ^-_1\hspace*{2.2mm} 3\hspace*{2.2mm}..\hspace*{2.2mm}1\hspace*{2.2mm}^-_2\hspace*{2.2mm}0&&\,1\hspace*{2.2mm} ^-_2\hspace*{2.2mm}0\hspace*{2.2mm} ^-_3\hspace*{2.2mm} 2\hspace*{2.2mm} ^-_1\hspace*{2.2mm} 3\hspace*{2.2mm}^-_2\hspace*{2.2mm}1\hspace*{2.2mm}^{..}_2\hspace*{2.2mm}0\\
\end{array}\end{eqnarray}
Instead of continuing one more step of such process, a different example of an exchange is given in the figure via a suggested 4-cycle with gray interior.

The lower sequence in Figure~\ref{noera} starts on its left with a $(4\times 4)$-bicutout $X$ containing a $(4\times 3)$-xcutout of a disjoint union $G$ of two  3-cubes. In it, a pair of 4-cycles with gray and light-blue interiors takes, via corresponding exchanges and to the right of the leftmost ``$\rightarrow$", to the $(4\times 4)$-bicutout $X'$ of a planar graph $G'$. About it, a 4-cycle $C=C_0\cup C_1$ with yellow interior takes, via an exchange of $C_0$ by $C_1$ and to the right of the second ``$\rightarrow$", to the $(4\times 4)$  bicutout $X''$ of a toroidal graph $G''$. In it, a suggested 4-cycle takes, to the right of the third ``$\rightarrow$", to an extension $X'''$, namely the bicutout of a 2-toroidal graph $G^{**}$, so its genus is 2, see Example~\ref{lastrem}. 

If we do not extend $X''$ to $X'''$, then a handle attached by the suggested exchange fails to raise the genus. This is exemplified on the center-left of Figure~\ref{ahiva}, where an exchange is indicated in dashed tracing in a 6-cutout $Y$ equivalent to  $X''$, 
with a redrawing of the two edges produced by the exchange traversing the tilted pair of opposite sides of $Y$ and identification of all opposite sides of $Y$. So, this still yields a graph of genus 1, call it  $G^*\neq G''$. Colors gray, yellow and light-blue are added to three resulting faces for ease of recognition. 
 
On the right side of Figure~\ref{ahiva}, an $(8\times 4)$-bicutout $Y^{**}$ formed by the union of two $(4\times 4)$-bicutouts $Y^*_1,Y^*_2$, each equivalent to $Y$, is presented, representing $G^{**}$, see Example~\ref{lastrem}.   
On the other hand, the extension graph $G^{**}$ in its $(8\times 4)$-bicutout $X'''$, admits more exchanges, taking it to graphs with genera $>2$. 
\end{example}

\subsection{Raising the genus of orientable surfaces}

\begin{example}\label{satan}
A case of raising the genus one unit via an exchange is given on the left side of Figure~\ref{paquita}, where the graph $G^{**}$ on the lower right of Figure~\ref{noera} is redrawn in a bicutout $X''''$ translated from the bicutout $X'''$. Notice $X'''$ has two 10-belts M$'$ and N$'$ with the interior of their faces
 in yellow and light-blue, respectively,
having each just an end-vertex of each of the two crossover green (color 3) edges. Each of M$'$ and N$'$ is formed by the union of two faces of $M(G'')$ with 3-permutation colors 321 and 132, produced by the removal of the red edges of $C_0$.
\end{example}

\begin{theorem}\label{finally}\cite[Theorem 29]{+1}
Let $X$ be a cutout of a $\Pi$-map $M(G)$ of a simple cubic graph $G$ of girth 4. Let $C$ be a 4-cycle $C=C_0\cup C_1$ with $C_0,C_1$ as in (\ref{c0c1}) and $C_1$ crossing edges of $G\setminus C_0$ in all cutouts of orientable surfaces realizing the genus of $G$. Let $G\setminus E(C_0)$ be connected. Then, each exchange in $X$ takes to a connected simple cubic graph $G'$ of girth 4 whose genus is larger than that of $G$. If $G$ has an ETGC, then $G'$ has an ETGC and a PFC corresponding to that of $G$. 
\end{theorem}

\begin{proof} Let $S^1$ denote the circle obtained by identifying the ends 0 and 1 of the real interval $[0,1]$. A handle $H$ containing two green edges as in Example~\ref{satan} is given by the image of a homeomorphism $\Theta:S^1\times[0,1]\rightarrow X''''$ with $\Theta(S^1\times\{0\})$ equal to M$'$ and $\Theta(S^1\times\{1\})$ equal to N$'$, using the notation of Example~\ref{satan} and $X''''$ in place of $X$. By removing the interiors of the faces of N$'$ and M$'$ and identifying the borders of $\Theta(S^1\times\{0\})$ and $\Theta(S^1\times\{1\})$ of $\Theta(S^1\times[0,1])\equiv H$ with the belts M$'$ and N$'$ via $\Theta$, respectively, an embedding of $G^{**}$ into a surface of larger genus is obtained, like $\mathbb{T}_2$ of genus 2, a 2-toroid, in the case of Example~\ref{satan}. 
These treatments generalize to the general situation of the theorem, where two $\ell$-cycles with $\ell\equiv 2 \mod 4$, like M$'$-N$'$, are identified with the circle borders of a handle.
\end{proof}

\begin{figure}[htp]
\includegraphics[scale=0.57]{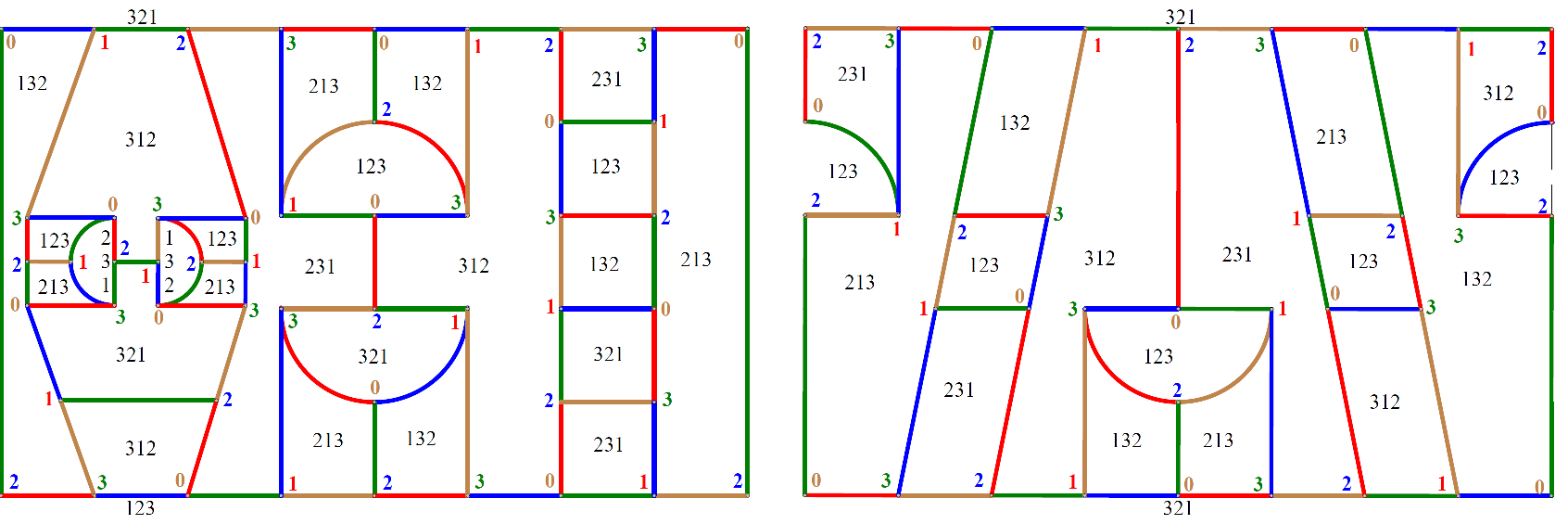}
\caption{Additional cases in Example~\ref{cacho}. Colors as in (\ref{colors}).}
\label{ideal-3}
\end{figure}

\begin{example}\label{hoyo} In continuation of Example~\ref{satan}, note that the handle
$H$ contributes two new faces in replacement of the faces with 10-belts M$'$ and N$'$ (see Figure~\ref{paquita}): One with 16-belt M and the other one with 8-belt N, whose respective color cycles are 
$$(2\,_-^1\,3\,_-^2\,0\,_-^3\,1\,_-^0\,2\,_-^1\,3\,_-^2\,0\,_=^3\,1\,_-^0\,2\,_-^1\,3\,_-^2\,0\,_-^3\,1\,_-^0\,2\,_-^1\,3\,_-^2\,0\,_=^3\,1\,_-^0)=(2\,_-^1\,3\,_-^2\,0\,_=^3\,1\,_-^0)^4$$ counterclockwise from the upper-right of 10-belt M$'$ and $$(3\,_-^0\,2\,_-^1\,0\,_=^3\,1\,_-^2\,3\,_-^0\,2\,_-^1\,0\,_=^3\,1\,_-^2)=(3\,_-^0\,2\,_-^1\,0\,_=^3\,1\,_-^2)^2$$ clockwise from the lower-right of 10-belt N$'$, respectively, where $_=^3$ indicates the edges implicated in the exchange. The edges bordering yellow areas and the tilted green edges induce the 16-belt M, while the edges bordering the gray areas and the tilted green edges induce the 8-belt G. The right side of Figure~\ref{paquita}, mentioned in Example~\ref{lastrem}, offers an embedding of $G^{**}$ into an octagonal cutout of $\mathbb{T}_2$, again showing that $G^{**}$ is 2-toroidal.  
\end{example}

\subsection{Exchanges via UETGCs}

\begin{example}\label{tor}
The middle-center of Figure~\ref{torcuato} contains a $(9\times 1)$-xcutout consisting of two component $(4\times 1)$-xcutouts of respective 3-cubes $G_0$ (left) and $G_1$ (right) separated by a 4-cycle $C$ with gray interior. 
This takes, via up and down arrows, to two copies $X$ and $Y$ of a $(9\times 1)$-xcutout of a connected simple planar cubic graph $G$ of girth 4 with $\ell$-belts only for $\ell\equiv 0 \mod 4$, depicted also to the right of the horizontal arrow as the image of a 3-dimensional union of the 3-cubes $G_0$ and $G_1$ with the pair $C_1\in E(C)$ joining them. No PFC as in previous figures are seen in this figure, as the shown 8-belts do not present the periodic patterns of (\ref{reverse}) with $x=d$ that guarantee such PFCs, so for now we are in the presence of a UETGC. This is corrected to DETGC or RETGC in Figure~\ref{bluyelow}.

\begin{eqnarray}\label{march}\begin{array}{llll}
2\hspace*{2.2mm} _-^0\hspace*{2.2mm}3\hspace*{2.2mm}_-^1\hspace*{2.2mm}0\hspace*{2.2mm}_-^2\hspace*{2.2mm}1\hspace*{2.2mm}_-^3\hspace*{2.2mm}2&&&
2\hspace*{2.2mm} _-^0\hspace*{2.2mm}1\hspace*{2.2mm}_-^3\hspace*{2.2mm}0\hspace*{2.2mm}_-^2\hspace*{2.2mm}3\hspace*{2.2mm}_-^1\hspace*{2.2mm}2\\
\;|\hspace*{5.7mm}_2|\hspace*{6.1mm}_3|\hspace*{6mm}_0|\hspace*{8mm}|&&&
\!_3|\hspace*{6mm}_2|\hspace*{6mm}_1|\hspace*{6mm}_0|\hspace*{6mm}_3|\\
\;|\hspace*{7.2mm}1\hspace*{2.2mm} \frac{0}{3}\hspace*{2.1mm}\,^2_0\hspace*{2.2mm} \frac{1}{2}\hspace*{2.2mm}3\hspace*{7.2mm}|&&&
\;_2^0\hspace*{2.2mm}\frac{1}{0}\hspace*{2.2mm}3\hspace*{7.3mm} |\hspace*{7.2mm} 1\hspace*{2.2mm}\frac{2}{3}\hspace*{2.4mm}_2^0 \\
\!_1|\hspace*{15mm}_1|\hspace*{15mm}_1|&&&
\!_1|\hspace*{16.7mm}|\hspace*{15mm}_1|\\
\;|\hspace*{7.2mm}3\hspace*{2.2mm} \frac{0}{1}\hspace*{2.2mm}\,^2_0\hspace*{2.2mm} \frac{3}{2}\hspace*{2.2mm} 1\hspace*{7.2mm}|&&&
\;|\hspace*{7.2mm}3\hspace*{2.2mm} \frac{0}{1}\hspace*{2.2mm}\,^2_0\hspace*{2.2mm} \frac{3}{2}\hspace*{2.2mm} 1\hspace*{7.2mm}|\\
\;|\hspace*{5.7mm}_2|\hspace*{6mm}_3|\hspace*{6mm}_0|\hspace*{8.1mm}|&&&
\;|\hspace*{5.7mm}_2|\hspace*{6mm}_3|\hspace*{6mm}_0|\hspace*{8.1mm}|\\
2\hspace*{2.2mm} _0^-\hspace*{2.2mm}3\hspace*{2.2mm} _1^-\hspace*{2.2mm} 0\hspace*{2.2mm} _2^-\hspace*{2.2mm} 1\hspace*{2.2mm} _3^-\hspace*{2.2mm} 2&&&
2\hspace*{2.2mm} _0^-\hspace*{2.2mm}3\hspace*{2.2mm} _1^-\hspace*{2.2mm} 0\hspace*{2.2mm} _2^-\hspace*{2.2mm} 1\hspace*{2.2mm} _3^-\hspace*{2.2mm} 2\\
\end{array}\end{eqnarray}

In the figure, the light-blue and yellow shadows are indicating the upper and lower 4-cycles in $X$ and $Y$, in order to locate them in the alternate representations $X'$ and $Y'$ of $X$ and $Y$ on the left half of Figure~\ref{bluyelow} (see them also schematized in  (\ref{march})). This allows unfoldings like those two depicted as xcutouts on the right side of the figure. By the way, in Figure~\ref{bluyelow}, PFCs as in Figure~\ref{torcuato} equivalent to depicted RETGCs appear due to the fact that all color cycles for the shown belts satisfy (\ref{reverse}) with $x=d$.
 In  (\ref{march}), 4-belts presented as semicircles on the left half of Figure~\ref{bluyelow} (in place  of rhombs) appear horizontally compressed, with some opposite vertices shown in small type.

Note that among others, (e.g. s (\ref{tess}), (\ref{octaedro}) and (\ref{march})), $G$ is not 3-edge connected, (but Example~\ref{klein} using a special case of the amalgam tool in Definition~\ref{amalgam} yields a toroidal cubic map of girth 4 with $\ell$-belts having $\ell\equiv 0 \mod 4$ not admitting ETC's, which inspires Definition~\ref{bicu} of toroidally 3-edge-connected graph and Conjecture~\ref{toroid}).
\end{example} 

\begin{figure}[htp]
\includegraphics[scale=0.57]{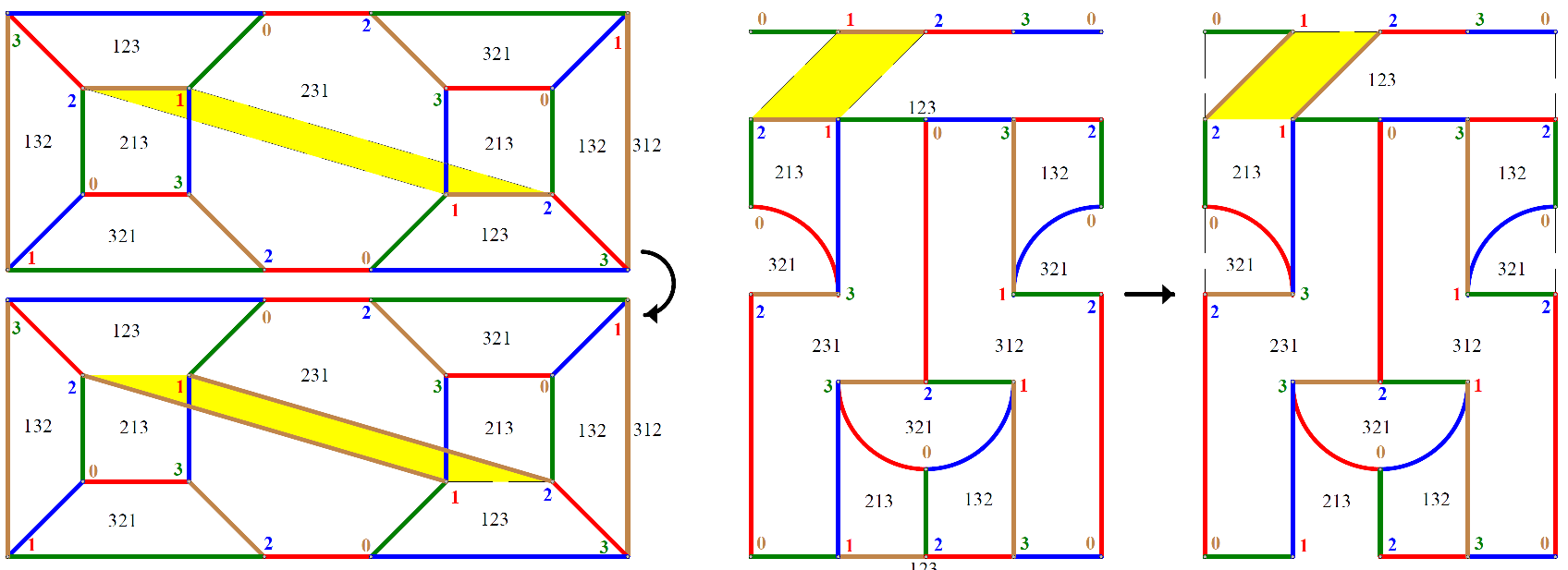}
\caption{Taking xcutout  of planar map into bicutout of  toroidal map. Colors as in (\ref{colors}).}
\label{sqsq}
\end{figure}

\begin{example}\label{cacho}
Figure~\ref{ideal-3} applies the unfolding tool to the two left representations in Figure~\ref{bluyelow}. Associated PFCs
$\mu$ equivalent to depicted RETGCs $\mu'$ as in Subsection~\ref{face} are indicated.
\end{example}

\begin{example}\label{sq}
The right side of Figure~\ref{sqsq} shows how to transform the xcutout $Y'$ in Example~\ref{tor} into the bicutout of a toroidal graph by adding an upper extra copy of the bottom 4-path in $Y'$ together with an exchanging 4-cycle represented as a yellow rhomb. On the left of the figure, an alternate representation of such transformation is given. In both cases, associated PFCs $\mu$ equivalent to shown RETGCs as in Subsection~\ref{face} are indicated.
Additionally, Figure~\ref{idea-24} shows extensions of both $X'$ and $Y'$ with corresponding alternate 3-dimensional representations of the resulting graphs shown to their  right. By rotating a counterclockwise right angle both $X'$ and $Y'$, and applying Definition~\ref{amalgam} to the resulting cutouts, Figure~\ref{vertical} shows the first stages of what is obtained. 
\end{example}

\begin{figure}[htp]
\includegraphics[scale=0.57]{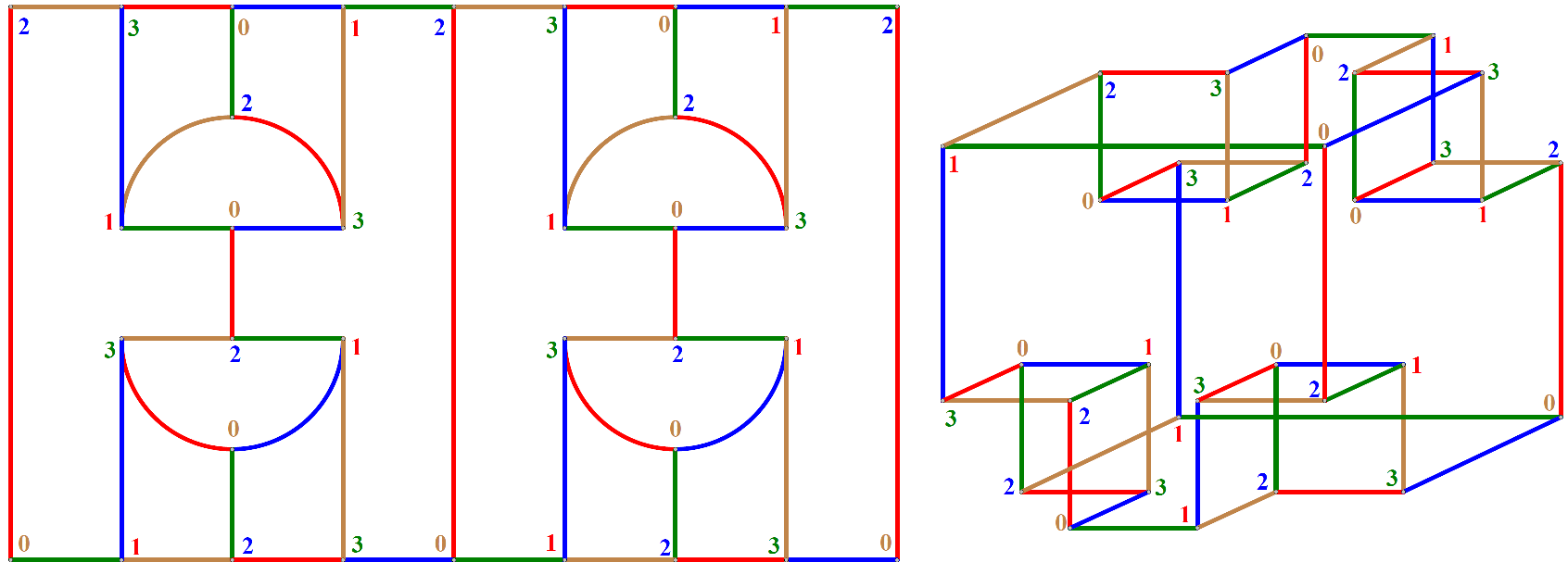}

\includegraphics[scale=0.57]{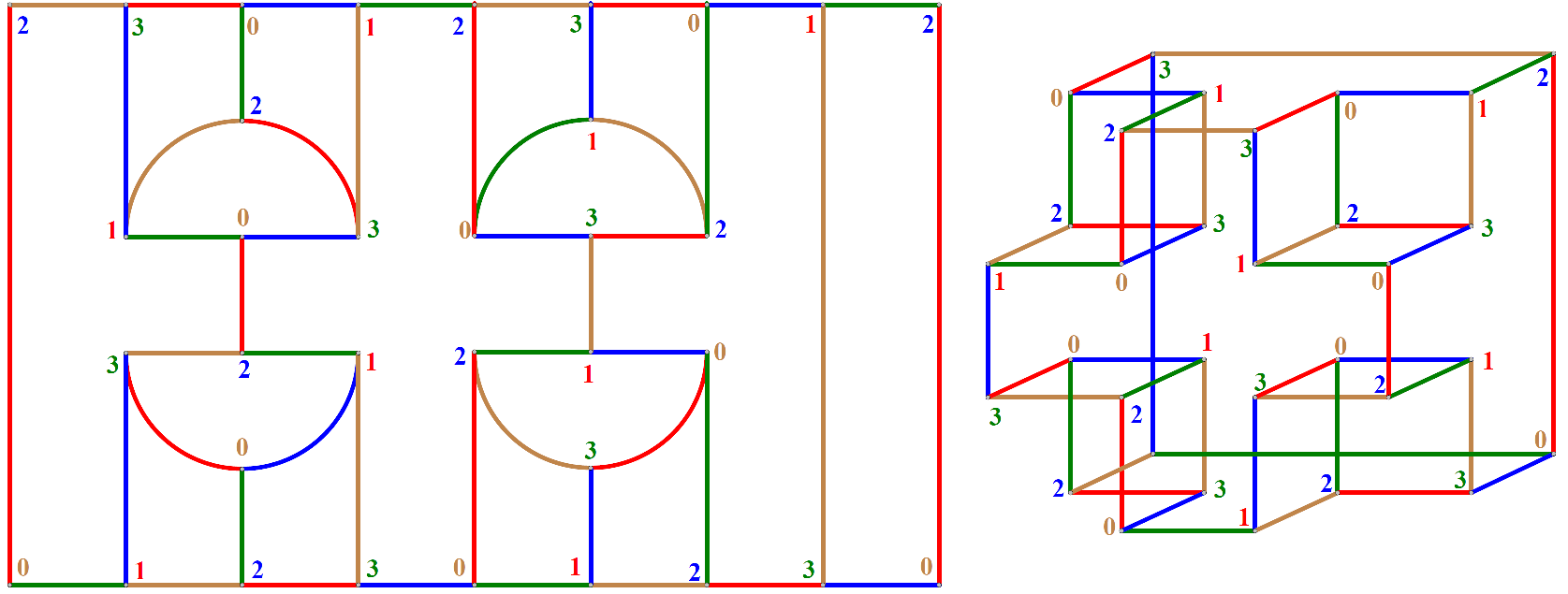}

\includegraphics[scale=0.57]{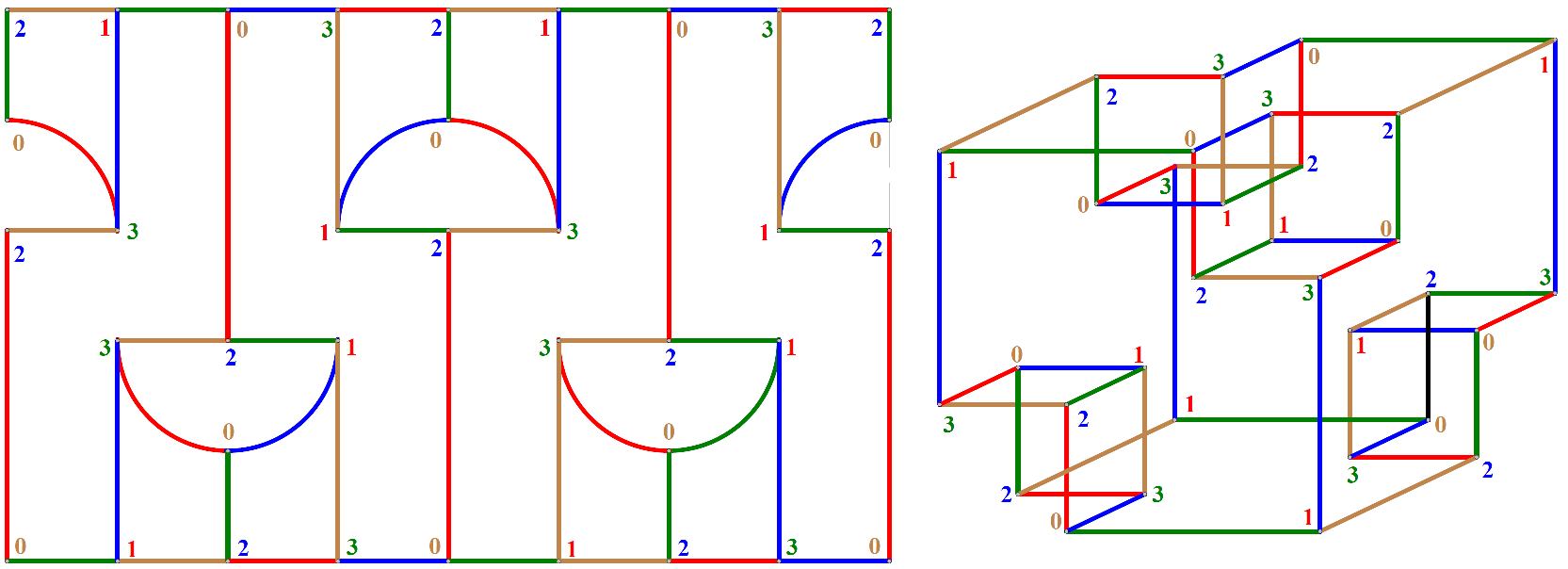}
\caption{Three pairs of extensions of the left cutout in Figure~\ref{bluyelow}. Colors as in (\ref{colors}).}
\label{idea-24}
\end{figure}

\begin{figure}[htp]
\hspace*{6mm}
\includegraphics[scale=1.4]{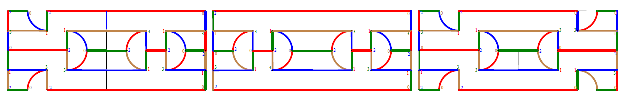}
\caption{Three amalgams applying Definition~\ref{amalgam}. Colors as in (\ref{colors}).}
\label{vertical}
\end{figure}

\section{Amalgams}

\begin{definition}\label{amalgam}
Let  $X_1,X_2$ be cutouts of respective simple planar cubic graphs $G_1,G_2$ of girth 4. Let  $C_i\subset G_i$ be a cycle given by the interval identification 
$[(0,0),(x_0,0)]\equiv[(0,y_0),(x_0,y_0)]$ in parallel in $X_i$, for $i=1,2$. Let $C_1\equiv C_2$, meaning there is a graph isomorphism $\rho:C_1\rightarrow C_2$ such that $\rho(v_i^j)=v_i^j$, for $i=1,2$ and $j=0,1$, where $v_i^0\in V(C_i)$ corresponds to $(0,0)\equiv(0,y_0)$ and $v_i^1\in V(C_i)$ corresponds to $(x_0,0)\equiv(x_0,y_0)$ before identification of the cutouts $X_i$ as in Definition~\ref{cut} so $C_1$ and $C_2$ keep the same orientation.
Then, an {\it amalgam} of $X_1$ to $X_2$ via $C_1\equiv C_2$ is a cutout $X$ of the simple planar cubic graph $G$ of girth 4 obtained by identifying the right side of $X_1$ in parallel to the left side of $X_2$, yielding $$(x_0,0)_{X_1}\equiv(0,0)_{X_2}, (x_0,y_0)_{X_1}\equiv(0,y_0)_{X_2}\mbox{ and }[(x_0,0),(x_0,y_0)]_{X_1}\equiv[(0,0),(0,y_0)]_{X_2}$$ with $E(C_1)\equiv E(C_2)$ absent in $G$ and each path $(v_1,v',v_2)$ having $v_i\in V(G_i)\setminus C_i$, ($i=1,2$), and $v'\in C_1\equiv C_2$ obtained by identifying 
vertices $v'_1\in V(C_1)$ and $v'_2\in V(C_2)$, with $v_1v'_1\in E(G_1)$ and $v'_2v_2\in E(G_2)$ replaced by the edge $v_1v_2$ in $G$.
\end{definition}

\begin{definition}\label{bicu}
Let $G$ be a toroidal connected simple cubic $\Pi$-graph of girth 4. Let $X$ be a bicutout of a toroidal map $M(G)$. Then, $X$ can be interpreted as an xcutout $X'$ of a planar map $M(G')$
and as a ycutout $X''$ of a planar map $M(G'')$, where $G'$ and $G''$ are $\Pi$-graphs. If at least one of $G'$ and $G''$ is 3-edge-connected, then we say that $G$ is {\it toroidally 3-edge connected}. 
\end{definition}

\begin{figure}[htp]
\includegraphics[scale=0.57]{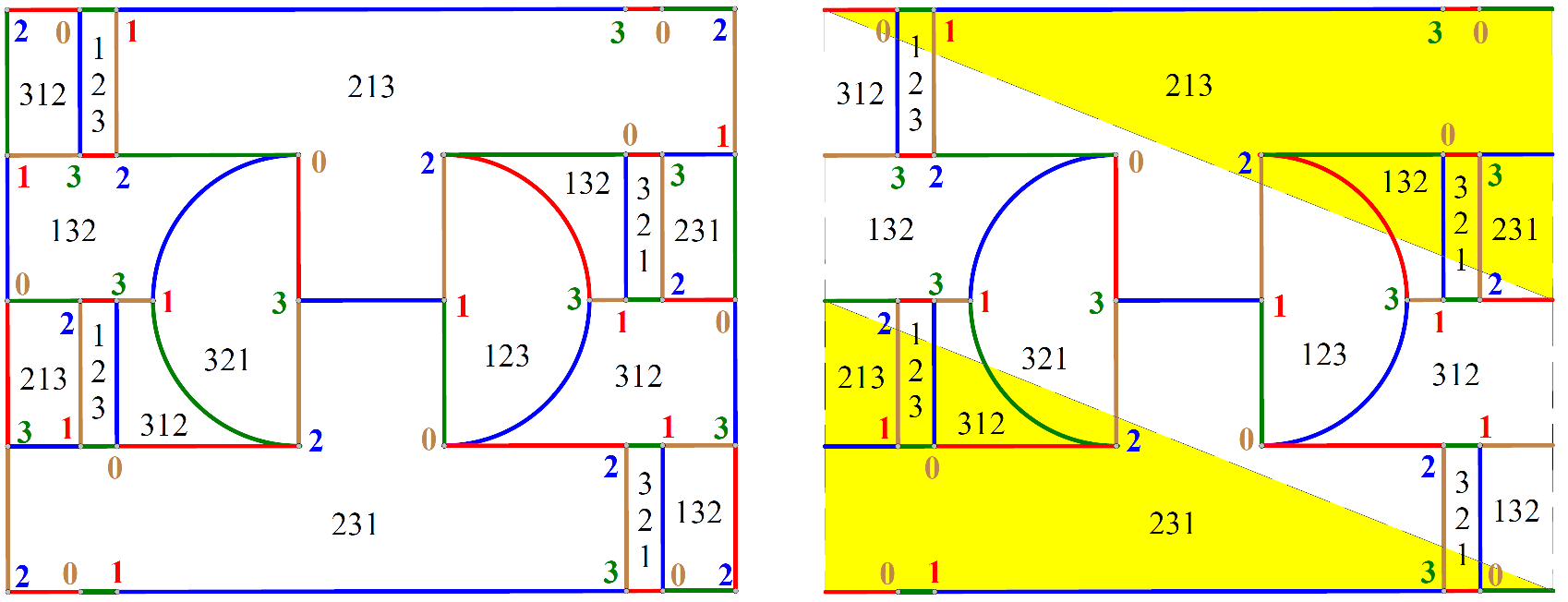}
\caption{Illustration for Example~\ref{klein} and Definition~\ref{bicu}. Colors as in (\ref{colors}).}
\label{kk}
\end{figure}

\begin{example}\label{klein}
The left side of Figure~\ref{kk} shows a right-angle rotated cutout, let us call it $X_1$ or $X_2$, of the left-side cutout in Figure~\ref{bluyelow}. To the right of it, the amalgam provided by the special case $X_1=X_2$ in Definition~\ref{amalgam} takes place yielding the bicutout $X$ of a graph $G$. However, interpreting from the deletable cycle $C_1\equiv C_2$ the half-edges on its left as continuations of the corresponding half-edges on its right produces a toroidal graph $G$ that does not admit an ETC, even though all $\ell$-belts of $G$ have $\ell\equiv 0 \mod 4$.
It seems that the condition in the following definition must be included in order to insure the existence of an ETC. On the other hand, translating the yellow upper triangle to have its upper border 4-cycle identified with the lower border 4-cycle of the yellow lower triangle makes the dashed-bordered rhomboid to be interpreted as a bicutout $X'$ containing a toroidal graph $G'$ with an RETGC and corresponding equivalent PFC.    
\end{example}

\begin{conjecture}\label{toroid}
Let $M(G)$ be a toroidal map whose 1-skeleton is a connected simple cubic $\Pi$-graph of girth 4. If $G$ is toroidally 3-edge connected, then $M(G)$ admits an ETGC.
\end{conjecture}

\section{Piercing}\label{p}

Recall that a graph has genus $g$ if it can be embedded in an orientable surface of genus $g$ but not in any orientable surface of genus less than $g$ .
An additional tool that we call {\it piercing}, whose objective is yielding a connected simple cubic graph of girth 4 and every genus $g$ having an ETGC $\mu'$ is illustrated 
on the left of Figure~\ref{20-gray} for a toroidal graph $G_1$ obtained from a 6-cutout $X_1$ by identifying each of the three pairs of opposite sides by parallel translation,
where a PFC $\mu$ equivalent to the shown RETGC $\mu'$ as in Subsection~\ref{face} is indicated. On the right of the figure, a representation of $G_1$ with the same $\mu'$ is given in a bicutout $\Upsilon$, however with its face belts not respecting strictly (\ref{reverse}) for $x=d$, so that no PFC is available in this disposition. 

\begin{figure}[htp]
\includegraphics[scale=0.57]{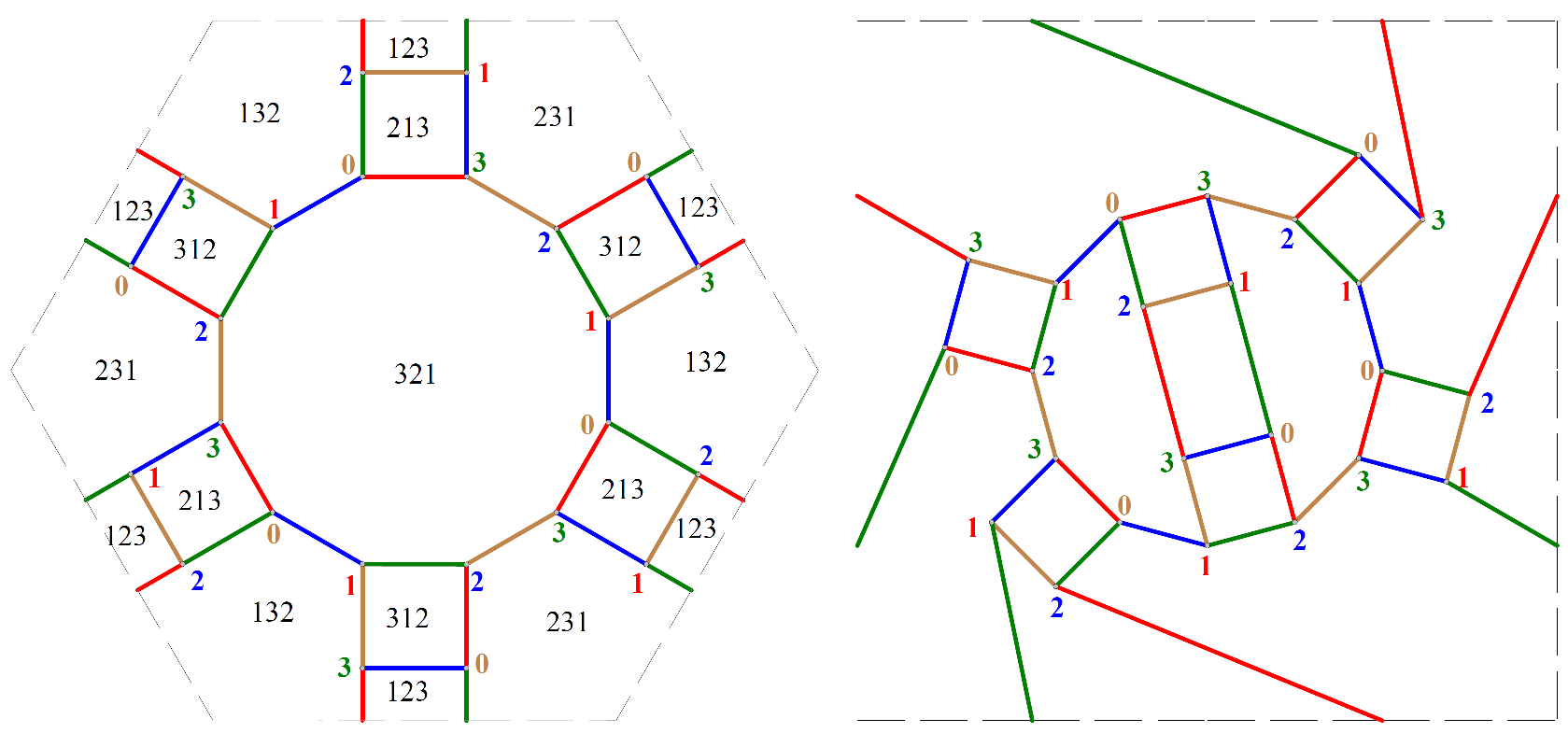}
\caption{Illustration for Section~\ref{p}. Colors as in (\ref{colors}).}
\label{20-gray}
\end{figure}

\begin{theorem}
There exists a family of $\Pi$-maps $M(G_i)$ of connected simple cubic graphs $G_i$ of girth 4 and genus $i\in\{1,2,\ldots\}\subset\mathbb{Z}$ with ETGCs $\mu'=\mu'(G_i)$  and associated  PFCs $\mu=\mu(M(G_i))$. Each such $M(G_i)$ is obtained by identifying opposite sides of a $(2i+1)$-cutout $X_i$ and has three $(8i+4)$-faces $F^i_0,F^i_1$ and $F^i_2$ and $6i+3$ 4-faces. Of these $6i+3$ 4-faces, $4i+2$ are adjacent, each, to each of the three $(8i+4)$-faces, while each of the remaining $2i+1$ 4-faces is adjacent on a pair of opposite edges to $F^i_1$ and $F^i_2$ and on the other pair of opposite edges to two 4-belts in the first $4i+2$ 4-belts. 
\end{theorem}

\begin{proof} Let $i\ge 1$. Let the $(2i+1)$-cutout $X_i=(L_1,L_2,\ldots,L_{2i+1},L_1^{-1},L_2^{-1},\ldots,L_{2i+1}^{-1})$, namely a $(4i+2)$-zonogon according to Definition~\ref{cutout}, be given by a regular $(4i+2)$-polygon. Then, $G_i$ is represented in $X_i$ starting with a central $(8i+4)$-belt $C_0^i$ realized as a regular $(8i+4)$-polygon in $X_i$ with alternate edges $e_1,\ldots,e_{4j+2}$ parallel to the corresponding sides $L_1,\ldots,L_{2i+1}^{-1}$ of $X_i$, that is: $e_j\parallel L_j\parallel e_{j+2i+1}\parallel L_j^{-1}$, $\forall j=1,\ldots, 2i+1$. Let us place a square 4-belt $B_j$ outside $C_0^i$ in $X_i$ that shares $e_j$ with $C_0^i$, $\forall j=1,\ldots,4i+2$. Let us make the disjoint union $B_j\cup B_{j+2i+1}$ of 4-belts be a subgraph of a copy $\bar{B}_j=B_j\cup B_{j,j+2i+1}\cup B_{j+2i+1}$ of $P_4\square K_2$ with middle 4-belt $B_{j,j+2i+1}$ straddling the sides-cancelling segment $L_j\equiv L_j^{-1}$ in the orientable surface $T^i$ of genus $i$ obtained from $X^i$ by identifications as in Definition~\ref{cutout} for $n=4i+2$. Let $e_i=(u_i,v_i)$ have end-vertices $u_i$ and $v_i$, for $i=1,\ldots,4i+2$, so that $C_0^i=(u_1,v_1,u_2,v_2,\ldots,u_{4i+2},v_{4i+2})$. 
Beside the face $F^i_0$ of $(8i+4)$-belt  $C_0^i$, there is a face $F^i_1$ with $(8i+4)$-belt $C^i_1=(v_1,u_2,w_2,z_{2+4i+2},v_{2+4i+2},u_{3+4i+2},\ldots,z_1)$  and a face $F^i_2$ with $(8i+4)$-belt $C^i_2=(v_3,u_3,w_4,z_{4+4i+2},v_{4+4i+2},u_{5+4i+2},\ldots,z_3)$, where $u_j$ and $v_j$ are adjacent to vertices $w_j$ and $z_j$ of $B_j\cup B_{j,j+2i+1}$ and $B_{j,j+2i+1}\cup B_{j+2i+1}$, respectively.
Each of the $2i+1$ copies $\bar{B}_j$ of $P_4\square K_2$ has 8 vertices, 10 edges and 3 faces. In addition, $M(G_i)$ has $4j+2$ edges (those in $(F^i_0\cap F^i_1)\cup(F^i_0\cap F^i_2)$. Thus, $|V(G)|=8(2i+1)$, $|E(G)|=10(2i+1)+4i+2$, $|F(M(G))|=3(2i+1)+3$, so the Euler characteristic of $M(G_i)$ is $$\chi(M(G_i))=(8(2i+1))-(10(2i+1)+4i+2)+(3(2i+1)+3)=-2i-1+3=2-2i$$ and the genus of $M(G_i)$ is $g(M(G_i))=i$. 
An associated PFC $\mu$ as in Definition~\ref{pipi} is given as follows: the face $F^i_0$ of belt $C_0^i$ is assigned 3-permutation 321, and its $(8i+4)$-belt is assigned associated clockwise color cycle $(0\,_-^1\,3\,_-^0\,2\,_-^3\,1\,_-^2)^{2i+1}$, where opposite pairs of edges with colors $(1\,_-^2\,0)$ and $(3\,_-^0\,2)$ separate $F^i_0$ from the faces $F^i_1$ and $F^i_2$. According to Subsection~\ref{face}, these faces receive colors 132 and 231 in $S_3$, respectively. This way, each copy $\bar{B_j}$ of $P_4\square K_2$ has central 4-belt face (separated in halves by the sides-cancelling segment $L_j\equiv L_j^{-1}$ in the border of $X_i$) with 3-permutation 123, and the copies of $P_4$ delimiting each of the $(4i+2)$ copies of $P_4\square K_2$ get color paths $(0\,_-^3\,2\,_-^1\,3\,_-^0\,1)$ and $(3\,_-^2\,1\,_-^3\,0\,_-^1\,2)$, respectively for $F^i_1$ and $F^i_2$.
\end{proof}

\section{PDS-complementation}\label{Messi}

\begin{definition}\label{abelardo} The {\it integer lattice graph} $\Gamma$ has vertex set $V(\Gamma)=\mathbb{Z}^2$ and is such that any two members of $V(\Gamma)$ are adjacent if and only if their Euclidean distance is 1. 
\end{definition}

\begin{theorem}\label{assume}
Consider a PDS $\Sigma$ of $\Gamma$ and let $\Lambda=\Gamma\setminus\Sigma$ be the complementary graph of $\Sigma$ in $\Gamma$. Let $M(\Lambda)$ be the planar map of $\Lambda$ induced by the standard embedding  of $\Gamma$ in $\mathbb{R}^2$. Assume that all belts of $\Lambda$ have lengths divisible by 4. In other words, assume that the induced components of $\Gamma\setminus\Lambda$ have vertex sets only of odd cardinalities. Under such assumptions, $\Lambda$ has DETGCs and RETGCs that yield corresponding PFCs of $M(\Lambda)$.
\end{theorem}

\begin{proof}
Propagation of spray (Definition~\ref{sabado}) acts as an algorithm on $M(\Lambda)$ through an ordering of its belts, where each subsequent belt is adjacent to an earlier treated belt. Then, application of Theorem~\ref{tt} is insured with all instances of  (\ref{reverse}) having $x=d$. But observe that all clusters in $M(\Lambda)$ can use reversed unfolding, or {\it folding}, to reduce $M(\Lambda)$ to the simplest case of PDS-complementation, namely the leftmost case in Figure~\ref{abel}, that coincides with that of Example~\ref{truncated}. This establishes our claims.
\end{proof}

\begin{corollary}\label{modorra}
If $M(\Lambda)$ is periodic both horizontally and vertically in $\mathbb{R}^2$, then $\mathbb{R}^2$ is partitioned into bicutouts given by parallelograms whose sides concatenate into equidistant lines in each of  two directions. Such lines have intersections yielding a plane lattice. By identifying the opposite sides of one such bicutout, a toroidal quotient $M(G)$ of a connected simple cubic graph $G$ of girth 4 is obtained that inherits DETGCs and their orthogonal RETGCs from those in $M(\Lambda)$.
Moreover, these ETGCs  provide corresponding PFCs.
\end{corollary}

\begin{figure}[htp]
\includegraphics[scale=0.57]{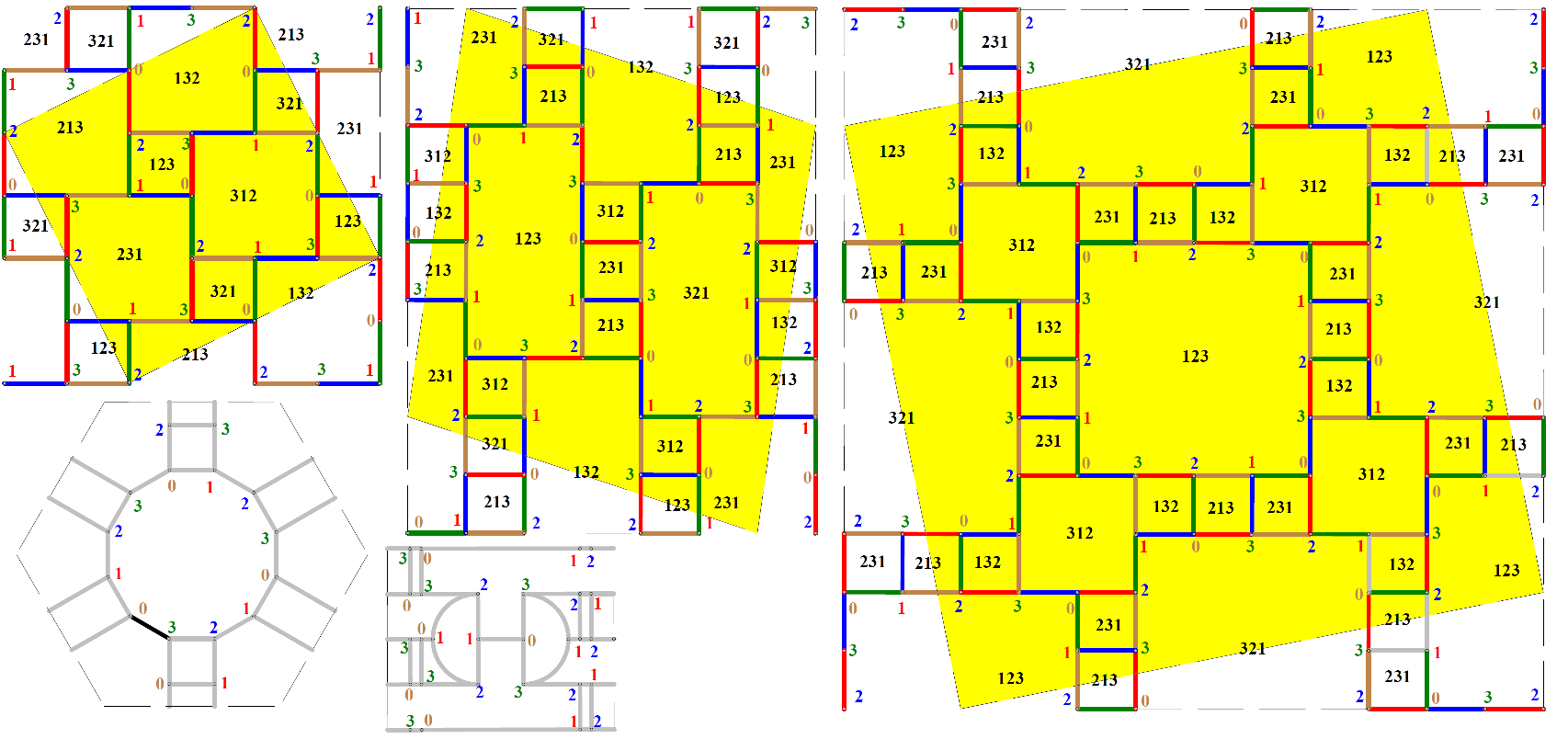}
\caption{Three DETGCs via perfect dominating sets and a non-$\Pi$-map}
\label{abel}
\end{figure}

\begin{example}
Figure~\ref{abel} contains, starting at its top, three examples of Corollary~\ref{modorra}. The leftmost example arises from a PDS whose induced components are isolated vertices, so it is a EDS. A yellow fundamental region is indicated yielding by identification of opposite sides a toroidal map with 16 vertices, 24 edges, four 8-belts and four 4-belts, so
$|V(G)|=16, |E(G)|=24, |F(M(G))|=8$ and $\chi(M(G)))=16-24+8=0$, so $g(M(G))=1$. This example is equivalent to that of Example~\ref{trunc0}.
The middle example arises from a PDS whose induced components are vertical copies of $P_3$. However, extensions as in Section~\ref{extensions} reduce this case to the leftmost one.
The rightmost example arises from a PDS whose induced components are copies of $P_3\square P_3$ interspersed with isolated vertices. It has 64 vertices, 96 edges, four 16-belts, four 8-belts and 24 4-belts, so $|V(G)|=64, |E(G)|=96, |F(M(G))|=32$ and $\chi(M(G)))=64-96+32=0$, so $g(M(G))=1$.
\end{example}

\section{Concluding remarks and conjectures}\label{toroidal}

\begin{remark}\label{remark}
The leftmost cutout $X$ in  (\ref{tess}) is related to the middle cutout in  (\ref{octaedro}) as follows. The leftmost vertical path of $X$ equals the rightmost vertical path and is joined with the second leftmost path by two horizontal edges colored $0\hspace*{2.2mm}^1_-\hspace*{2.2mm}2$ and $2\hspace*{2.2mm}^1_-\hspace*{2.2mm}0$. We apply {\it exchange} by replacing those two edges by the other two edges forming an auxiliary square $K_2\square K_2$ of $X$ in $\mathbb{R}^2$ but resulting in a cutout $X'$ equivalent to the middle cutout $X''$ in (\ref{octaedro}) (with the leftmost vertical path $[1_-^32_-^10_-^23_-^01]^t$ of $X''$ obtained as the second leftmost vertical path $[2_-^10_-^23_-^01_-^32]^t$ of $X'$) of a planar graph by the now apparently separated modified leftmost path $[0_-^12_-^31_-^03_-^20]^t$ of $X'$, since it is already present as the  now modified rightmost path of $X'$. 

Similar exchanges were given in the right-side cutout of (\ref{octaedro}), where horizontal-edge pairs indicated by diaeresis pairs are to be replaced by the other two edges in the auxiliary squares $K_2\square K_2$ shown to the right. 
\end{remark}   

\begin{eqnarray}\label{te}\begin{array}{ll|ll}
0\hspace*{2.2mm}_-^3\hspace*{2.2mm}1\hspace*{2.2mm}_-^0\hspace*{2.2mm}2\hspace*{2.2mm}_-^1\hspace*{2.2mm}3\hspace*{2.2mm}_-^2\hspace*{2.2mm}0&&&
0\hspace*{2.2mm}_-^3\hspace*{2.2mm}1\hspace*{2.2mm}_-^0\hspace*{2.2mm}2\hspace*{2.2mm}_-^1\hspace*{2.2mm}3\hspace*{2.2mm}_-^2\hspace*{2.2mm}0\\
\!_1|\hspace*{6mm}_2|\hspace*{23.5mm}_1|&&&
\!_1|\hspace*{6mm}_2|\hspace*{23.5mm}_1|\\
2\hspace*{2.2mm}_-^0\hspace*{2.2mm}3\hspace*{2.2mm}_-^1\hspace*{2.2mm}0\hspace*{2.2mm}_-^2\hspace*{2.2mm}1\hspace*{2.2mm}_-^3\hspace*{2.2mm}2&&&
2\hspace*{2.2mm}_-^0\hspace*{2.2mm}3\hspace*{2.2mm}_-^1\hspace*{2.2mm}0\hspace*{2.2mm}_-^2\hspace*{2.2mm}1\hspace*{2.2mm}_-^3\hspace*{2.2mm}2\\
:\hspace*{15mm}_3|\hspace*{6mm}_0|\hspace*{6mm}:&&&
:\hspace*{15mm}_3|\hspace*{6mm}_0|\hspace*{7.5mm}:\\
0\hspace*{2.2mm}_-^3\hspace*{2.2mm}1\hspace*{2.2mm}_-^0\hspace*{2.2mm}2\hspace*{2.2mm}_-^1\hspace*{2.2mm}3\hspace*{2.2mm}_-^2\hspace*{2.2mm}0&&&
0\hspace*{2.2mm}_-^3\hspace*{2.2mm}1\hspace*{2.2mm}_-^0\hspace*{2.2mm}2\hspace*{2.2mm}_-^1\hspace*{2.2mm}3\hspace*{2.2mm}_-^2\hspace*{2.2mm}0\\
&&&\!_1|\hspace*{6mm}_2|\hspace*{23.5mm}_1|\\
&&&2\hspace*{2.2mm}_-^0\hspace*{2.2mm}3\hspace*{2.2mm}_-^1\hspace*{2.2mm}0\hspace*{2.2mm}_-^2\hspace*{2.2mm}1\hspace*{2.2mm}_-^3\hspace*{2.2mm}2\\
&&&:\hspace*{15mm}_3|\hspace*{6mm}_0|\hspace*{7mm}:\\
&&&0\hspace*{2.2mm}_-^3\hspace*{2.2mm}1\hspace*{2.2mm}_-^0\hspace*{2.2mm}2\hspace*{2.2mm}_-^1\hspace*{2.2mm}3\hspace*{2.2mm}_-^2\hspace*{2.2mm}0\\
\end{array}\end{eqnarray}

\begin{example}\label{sondos}
The left case in  (\ref{te}) is obtained by setting $Q_3$ drawn as the 1-skeleton of a toroidal $\Pi$-map via a bicutout $X$ of $Q_3$.  
Note that the two shown 8-cycles in such $X$ are not chordless in $Q_3$. 
However, vertically stacked extensions of such $X$ represent graphs that as 1-skeletons of corresponding toroidal $\Pi$-maps have all their belts as chordless 4- and  8-cycles, as is the case depicted on the right side of the figure, with just one stacked extension of $X$. By the way, the coloring on the toroidal $\Pi$-map obtained from the bicutout $Y$ on this right side of (\ref{te}) is an ETGC, but does not have an associated PFC, since the four 8-belts do not respect (\ref{reverse}) with $x=d$. Horizontal and vertical concatenations of $Y$ still yield cubic $\Pi$-maps with ETGCs not extensible to face $S_3$-colorings. However, this graph can be drawn to have the $\Pi$-map presented in Example~\ref{truncated} and the left of (\ref{tess}).
\end{example}

\begin{figure}[htp]
\includegraphics[scale=0.58]{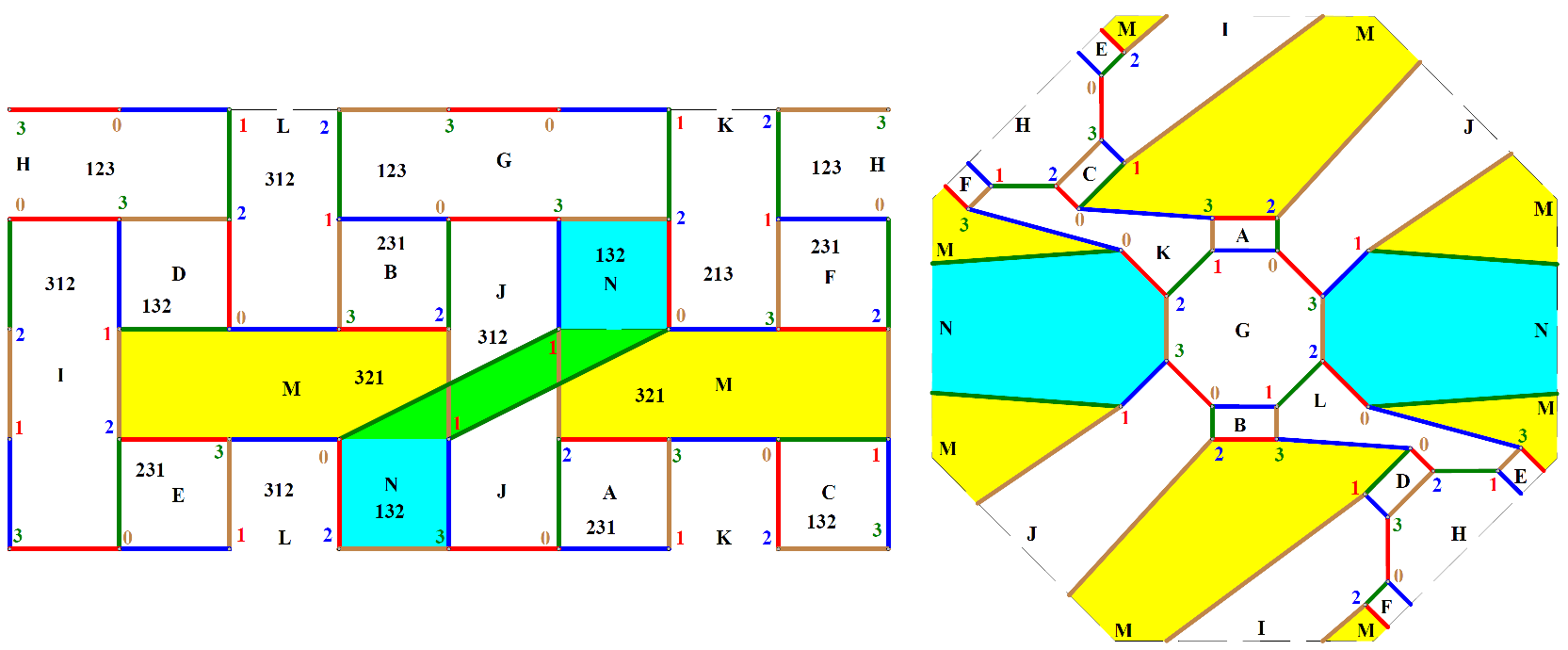}
\caption{Raising genus from 1 to 2 and octagonal cutout for $G^{**}$. Colors as in (\ref{colors}).}
\label{paquita}
\end{figure}

\begin{conjecture}\label{mango}
All connected simple cubic $\Pi$-maps $M(G)$ of girth 4 in genus-realizing orientable surfaces contain DETGCs (resp. RETGCs) with associated PFCs from which the original DETGCs (resp. RETGCs) can be recovered, allowing unique induced EGCs in the prisms $G\square K_2$. 
\end{conjecture} 

\begin{conjecture}\label{con1} ETGCs of maps of finite connected simple cubic graphs 
$G$ of girth 4 are obtained solely by means of the following seven constructive tools: spray (Definition~\ref{sabado}),
yielding the smallest such $G$, namely $G=Q_3$ (as in Corollary~\ref{under}),
extension (Definition~\ref{pe}),
unfolding (Definition~\ref{unfold}), exchange (Definition~\ref{exchange}), amalgam (Definition~\ref{amalgam}), piercing (Section~\ref{p}) and PDS-complementation (Section~\ref{Messi}). 
\end{conjecture}

\begin{question}\label{mapx}
Can Definition~\ref{bicu} and Conjecture~\ref{toroid} be generalized for the cases of $2g$-cutouts of $g$-toroidal graphs, for a notion of 
{\it $g$-toroidally 3-edge connected simple cubic graph}?
\end{question}

\begin{example}\label{lastrem} The right side of  Figure~\ref{paquita} shows a regular 8-polygon cutout $\bar{Y}$ of the 2-toroid $\mathbb{T}_2$ representing the graph $G^{**}$ in  Example~\ref{satan}, where the faces of the embedding are distinguished by means of the 14 capital letters from A to N. Notice the RETGC present in $\bar{Y}$.
On the left of the figure, the bicutout $Y^{**}$ on the right side of Figure~\ref{ahiva} is presented with those 14 letters indicating their position in the faces of $G^{**}$ with respect to $\bar{Y}$, where the 10-belts M$'$ and $G'$ with the two crossover green edges of handle $H$ in Example~\ref{satan} are redistributed into the 16-belt $M$ with yellow and (shared) light-green interior and the 8-belt N with light-blue and (shared) light-green interior, where the shared light-green parallelogram corresponds to the handle $H$ with visible top part say in the face of 16-belt M and hidden bottom part in the face of 8-belt N. Notice that $Y^{**}$ still holds here the original DETGC, so that $\bar{Y}$ should be flip, or turn over, in order to transform its RETGC into a DETGC as the one of $Y^{**}$. 
Noting that $G^{**}$ has 32 vertices, which is less than the 40 vertices of the graph $G_2$ on the right side of Figure~\ref{20-gray} in Section~\ref{p}, the following Question~\ref{alfin} is posed. 
\end{example}

\begin{figure}[htp]
\hspace*{1.6cm}
\includegraphics[scale=0.45]{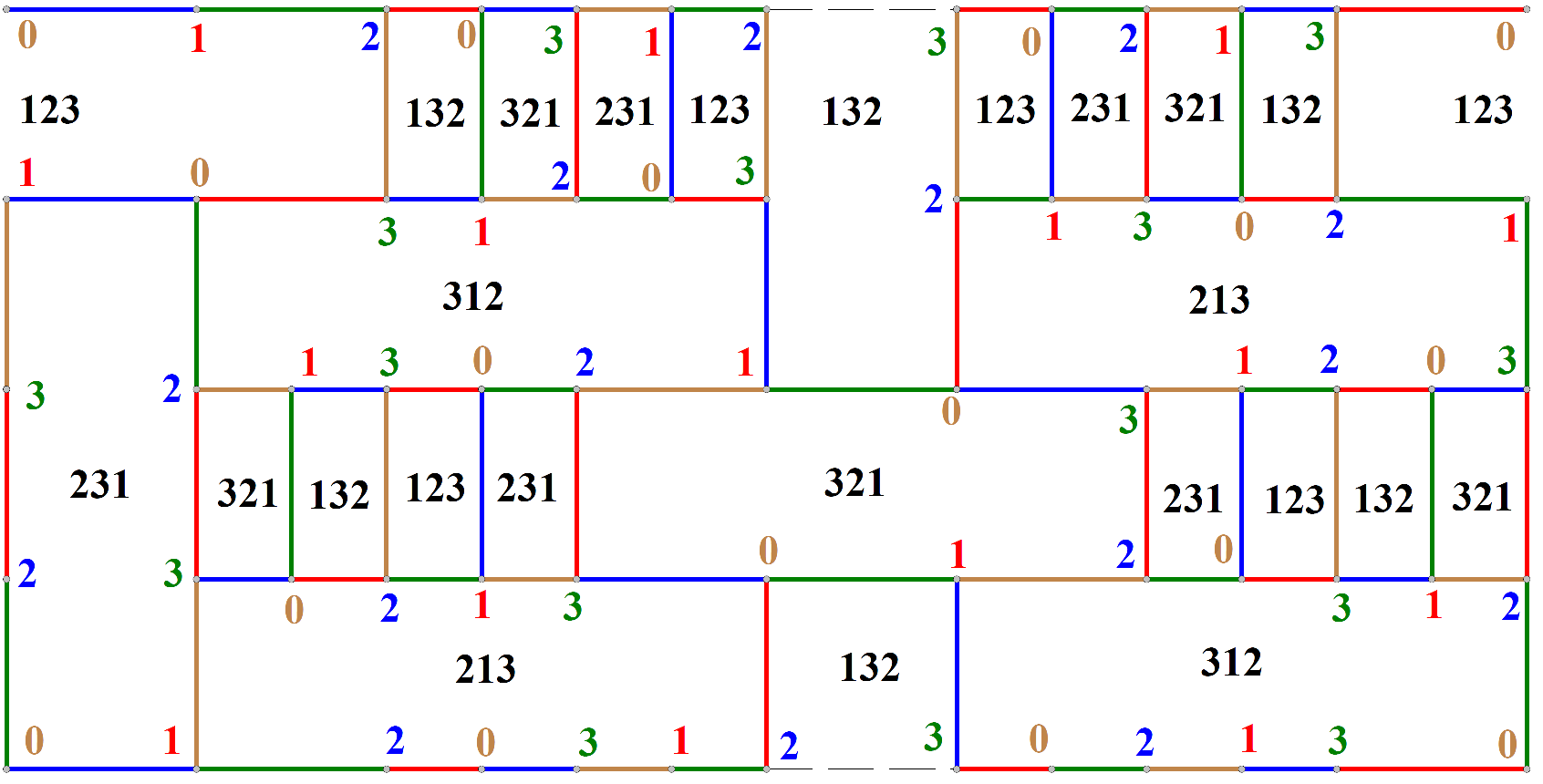}
\caption{Unfolding of $G^{**}$ on the bicutout $Y^{**}$ on the right of Figure~\ref{ahiva}. Colors as in (\ref{colors}).}
\label{esca}
\end{figure}

\begin{question}\label{alfin}
Is $G^{**}$ the smallest cubic graph of girth 4 with regular 8-polygon cutout representation in $\mathbb{T}_2$ and an ETGC? Moreover, we ask for the smallest cubic graph of girth 4 with ETGC embeddable in $\mathbb{T}_g$ via an adequate regular polygon, for any $g>0$.
\end{question}

\begin{remark}\label{observe} There are cubic $\Pi$-maps $M(G)$ of girth 4 with ETGCs in orientable surfaces realizing the genus, but that cannot have face $S_3$-colorings (Subsection~\ref{face}), as 
in the bicutout on the right of Figure~\ref{20-gray}, or the 8-belt $1\,_-^0\,2\,_-^1\,3\,_-^2\,0\,_-^1\;2\,_-^3\,1\,_-^2\,0\,_-^1\,3\,_-^2$ on the right of (\ref{te}), both cases with $x\neq d$. In such cases there is a form to obtain a $\Pi$-map of $G$ different from $M(G)$ with an associated face $S_3$-coloring. So, are there cubic $\Pi$-maps $M(G)$ of girth 4 with ETGCs that cannot in any way have $G$ embedded in a $\Pi$-map with face $_3$-coloring? 
 \end{remark}

\begin{remark}
The bicutout $Y^{**}$ of the toroidal $\Pi$-map $G^{**}$ on the right of Figure~\ref{ahiva} admits the unfolding shown in Figure~\ref{esca} and yielding, from the toroidal $\Pi$-map $M(G^{**})$, the unfolded toroidal $\Pi$-map $M(G^{**}_u)$ of a cubic graph $G^{**}_u$ of girth 4 that has the ETGC shown in the figure, together with an associated PFC. However, the half cutout $Y^*_1$ (or $Y^*_2$) containing the toroidal $\Pi$-map $M(G^*)$ admits
a similar unfolding, as indicated in the figure, but this does not have an ETGC. Note that this graph is not toroidally 3-edge connected, as in Definition~\ref{toroid}, as is also the case of the left or right half of $Y^{**}$, either half considered as a bicutout itself. However, one such $(4\times 4)$-bicutout yields a toroidal $\Pi$-map with an ETCD and associated PFC by identification along the dashed diagonal of $Y^{**}$t, which is not the case in Figure~\ref{esca} because of the presence of resulting 12-belts instead of 8-belts.   
\end{remark}

\begin{figure}[htp]
\includegraphics[scale=0.56]{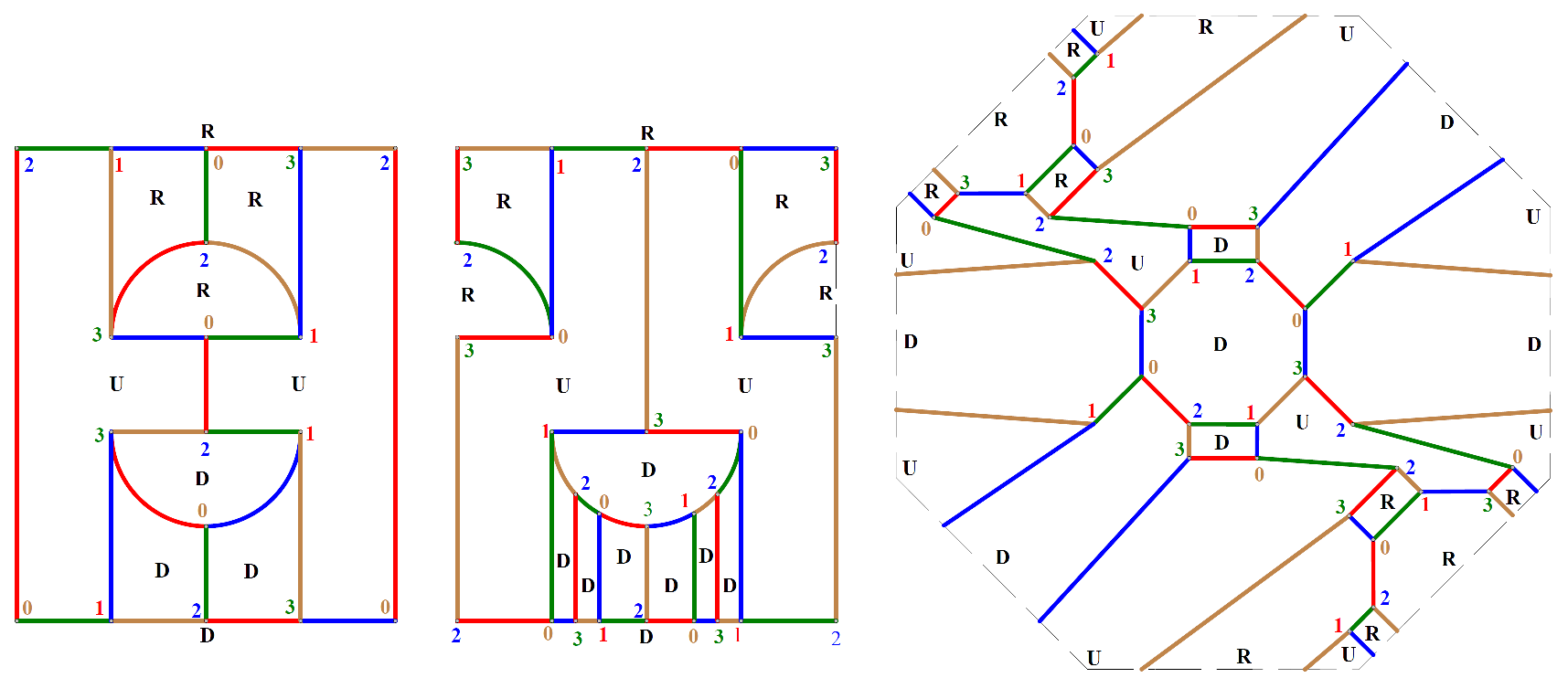}
\caption{Examples of ETGCs not inducing face $S_3$-colorings. Colors as in (\ref{colors}).}
\label{exper}
\end{figure}

\begin{example}\label{negating} Regarding Definition~\ref{opto},
Fiigure~\ref{exper} contains ETGCs in the leftmost and rightmost $\Pi$-maps of Figure~\ref{bluyelow} and in the right $\Pi$-map of Figure~\ref{paquita}. These ETGCs are presented to show that, by not following in their construction that s (\ref{reverse}) and (\ref{biarc}) are always applied with $x=d$, they become 
UETGCs. In fact, the faces of these $\Pi$-maps are labelled with the letters D, R and U according to whether their corresponding belts are directed, reversed or undirected, respectively.
In these UETGCs the subgraph $G_D$ induced by each cluster $D$ and the edges joining $D$ to Y-vertices have the belts enclosed in $G_D$ having a common letter D or R, so their belts are all directed or reverse. These {\it super-clusters} $G_D$ are mutually separated by faces necessarily labelled U. 
\end{example}

\begin{remark} In \cite{+1}, two cases of simple cubic graphs $G$ of girth 4 were considered, namely planar and toroidal graphs, that were extended subsequently to graphs embeddable in surfaces of larger genera. In the present work, apart from Questions~\ref{preg}, \ref{mapy}, \ref{mapz}, \ref{mapx} and~\ref{alfin},  
three 
conjectures are posed: First, Conjecture~\ref{toroid}, asserting that any simple cubic graph $G$ that is {\it toroidally 3-edge-connected} (see Definition~\ref{bicu}) with a $\Pi$-map $M(G)$ whose $\ell$-belts have $\ell\equiv 0 \mod 4$ has an ETGC. Second, 
Conjecture~\ref{mango}, asserting that all connected simple cubic $\Pi$-maps $M(G)$ of girth 4 in genus-realizing orientable surfaces having just $\ell$-belts for values $\ell\equiv 0$ mod  four contain ETGCs with associated PFCs from which the original ETGCs can be recovered, allowing unique induced EGCs in the prisms $G\square K_2$. Third,
updating \cite[Conj. 31]{+1} asserting that all existing ETGCs on connected simple cubic graphs of girth 4 are obtained from the smallest cubic graph of girth 4, namely the 3-cube, by applying four constructive tools, denoted here as {\it spray} (Definition~\ref{sabado}), {\it extensions} (Definition~\ref{pe}), {\it unfoldings} (Definition~\ref{unfold}) and {\it exchanges} (Definition~\ref{exchange}). 
Those tools are updated in the present work for the case of $\Pi$-maps. 
In fact, three more tools, namely {\it amalgam} (Definition~\ref{amalgam}), {\it piercing} (Section~\ref{p}) and {\it PDS-complementation} (Section~\ref{Messi}) are added to those four, resulting in Conjecture~\ref{con1}. 
\end{remark}

\section{Appendix: Additional Petrie dual cases}\label{appendix}

\begin{example}\label{gratis}
As an example of the PDS-complementation tool of Section~\ref{Messi}, consider the bicutout on the upper-left to Figure~\ref{abel}, with dispositions similar to those in the two lower thirds of Figure~\ref{schem} but now in  Figure~\ref{tort}. The directed Petrie graph itself can then be taken to have a bicutout obtained as a concatenation of two copies of the cutout in the center of the middle third in Figure~\ref{schem}.
\end{example} 

\begin{example}\label{cuero}
The same graph $G$ in Example~\ref{gratis} have alternative Y-map $M'(G)$ with bicutout as on the left side of display (\ref{te}). Figure~\ref{(18)der} depicts the four Petrie circuits in corresponding images of a $4\times 4$-bicutout whose vertices are denoted in hexadecimal notation in the integer interval $[0,f]$ with their colors obtained modulo 4.
The Petrie dual in the figure is formed by the four clockwise-oriented 12-circuits whose arcs share the colors of their central circle (as in the center of Figure~\ref{schem}), while their shown colors are those of their neighboring 12-circuits.
\end{example}

\begin{example}\label{bilu}
Figure~\ref{blupetri} does as in Example~\ref{gratis} for the left xcutout in Figure~\ref{bluyelow}, where (\ref{reverse}) happens always with $x=d$, so a FPC is available in this case. An alternative coloring on such map not complying with  $x\ne d$ throughout, as on the left of Figure~\ref{exper}, is shown in Figure~\ref{optam}.
\end{example}

\begin{example}\label{cuesta}
Figure~\ref{nxgtight} consists of four copies of the cutout on the right of Figure~\ref{paquita}. In each of the four cases, yellow color is employed to distinguish alternate regions separated by the directed circuits whose arcs bare a color in $[3]_0$.
\end{example}

\begin{figure}[htp]
\hspace*{8mm}
\includegraphics[scale=1.0]{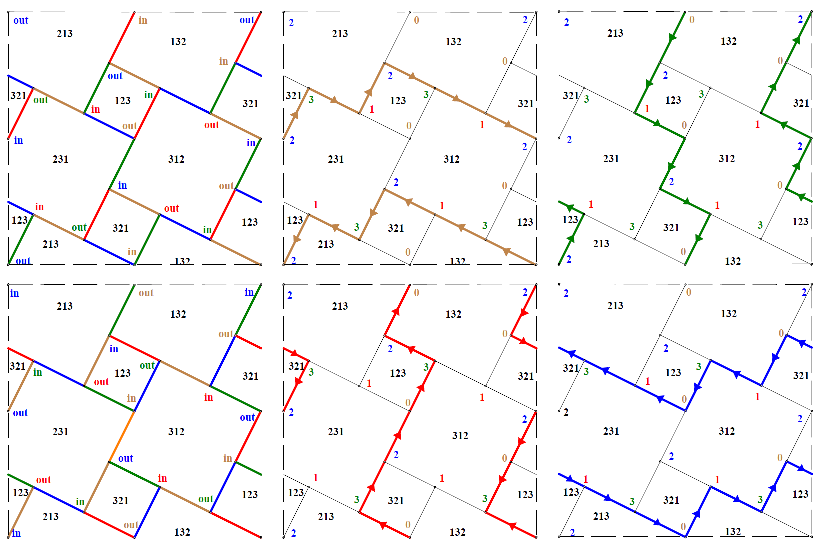}
\caption{Illustration for Example~\ref{gratis}. Colors as in (\ref{colors}).}
\label{tort}
\end{figure}

\begin{figure}[htp]
\includegraphics[scale=1.25]{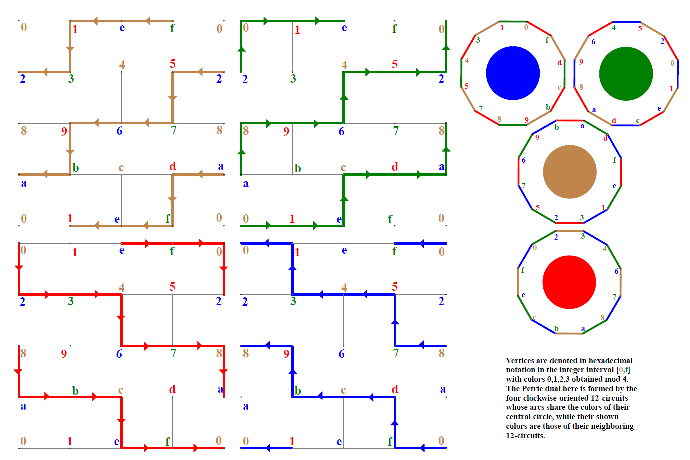}
\caption{Illustration for Example~\ref{cuero}. Colors as in (\ref{colors}).}
\label{(18)der}
\end{figure}

\begin{figure}[htp]
\includegraphics[scale=1.15]{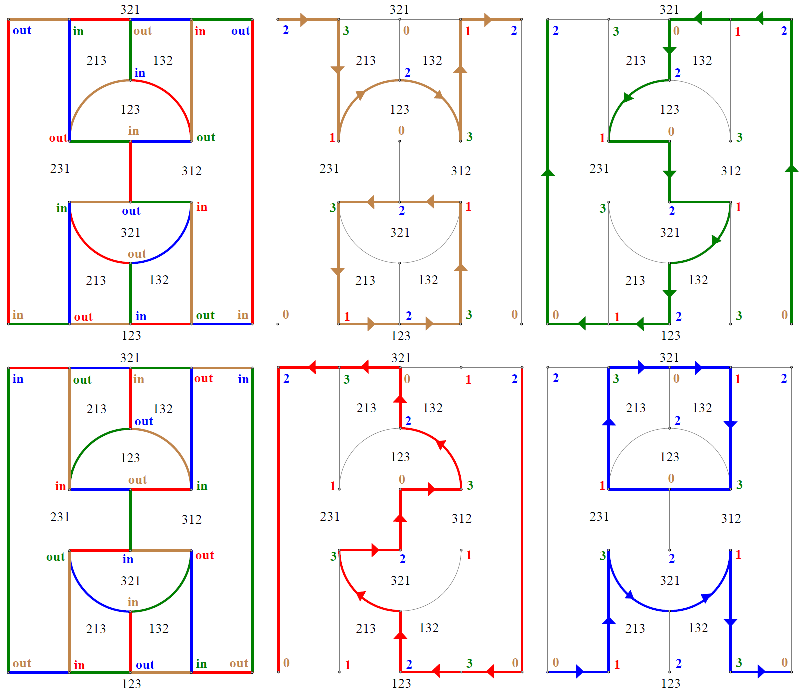}
\caption{Illustration for Example~\ref{bilu}, first part. Colors as in (\ref{colors}).}
\label{blupetri}
\end{figure}

\begin{figure}[htp]
\includegraphics[scale=0.55]{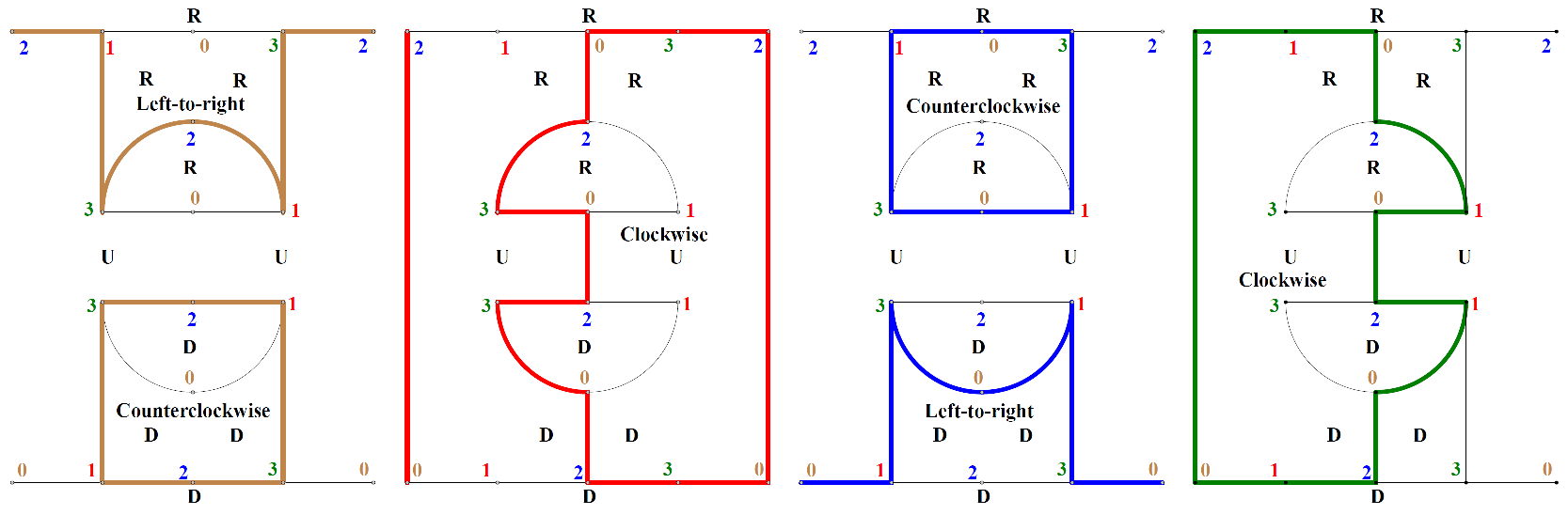}
\caption{Illustration for Example~\ref{bilu}, second part. Colors as in (\ref{colors}).}
\label{optam}
\end{figure}

\begin{figure}[htp]
\hspace*{1.5cm}
\includegraphics[scale=1.0]{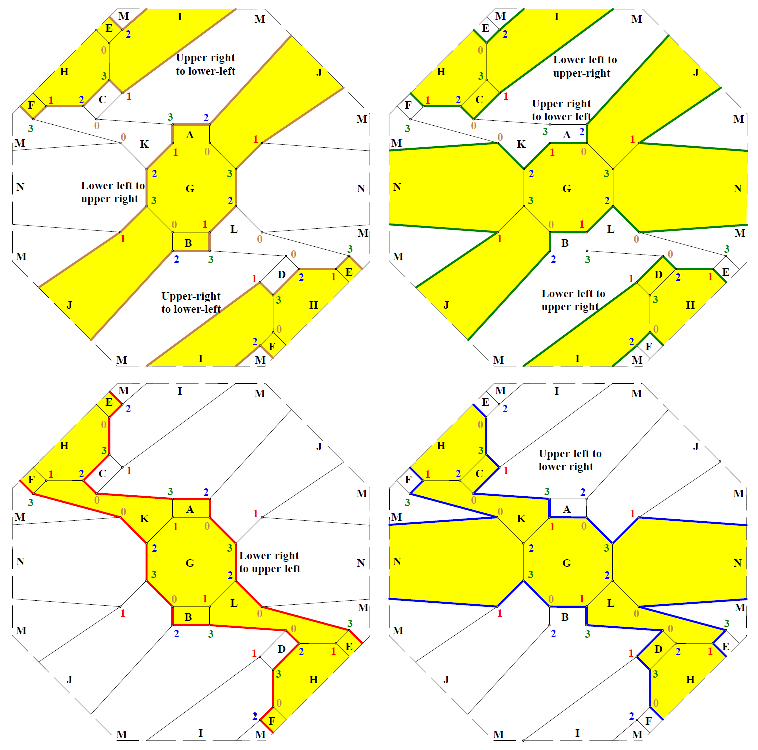}
\caption{Illustration for Example~\ref{cuesta}. Colors as in (\ref{colors}).}
\label{nxgtight}
\end{figure}

\begin{figure}[htp]
\hspace*{2.5cm}
\includegraphics[scale=0.7]{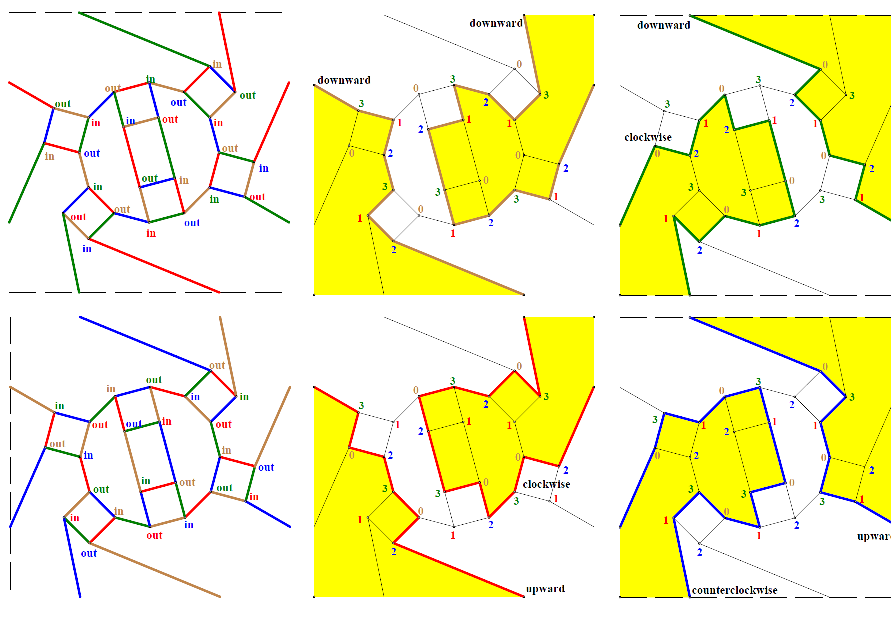}
\caption{Illustration for Example~\ref{fan}, first part. Colors as in (\ref{colors}).}
\label{68cuadro}
\end{figure}

\begin{figure}[htp]
\hspace*{2.0cm}
\includegraphics[scale=1.15]{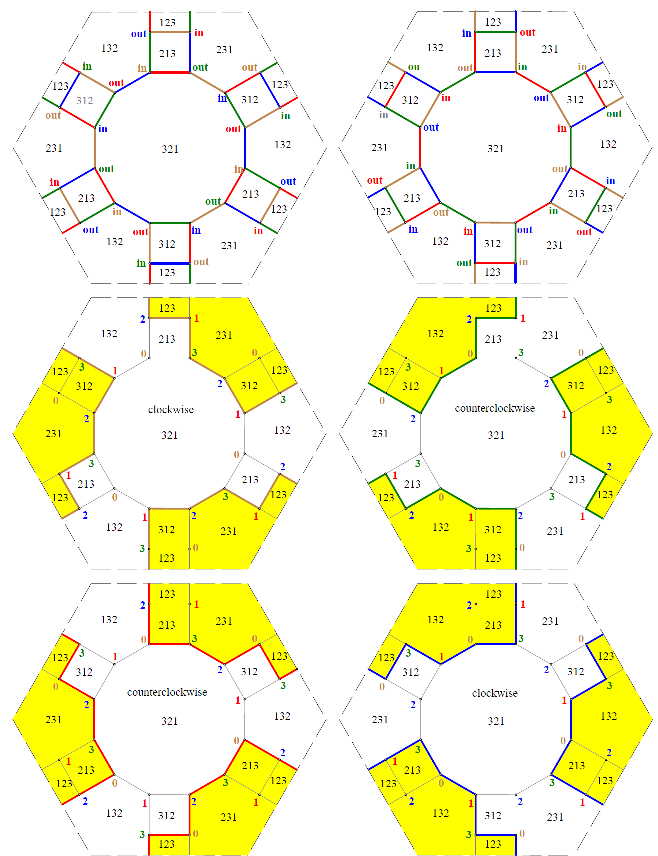}
\caption{Illustration for Example~\ref{fan}, second part. Colors as in (\ref{colors}).}
\label{68ojo}
\end{figure}

\begin{example}\label{fan}
Figure~\ref{68cuadro} depicts four bicutouts of the Petrie dual of the Y-map on the right of Figure~\ref{20-gray}, one per color of $[3]_0$, where yellow color separates the circutis, but in this case color yellow changes to white on continuation after reaching the horizontal and vertical sides of the bicutouts. The same is done for the 6-cutout
on the left side of Figure~\ref{68ojo}. However, yellow separation continues on passing the sides of the bicutout in this case.
\end{example}


\begin{thebibliography}{99}

\bibitem{Araujo} C. Araujo and I. J. Dejter, {\it Lattice-like total perfect codes}, Discussiones Mathematicae Graph Theory, {\bf 34} (2014), 57--74.

\bibitem{Horak} C. Araujo, I. J. Dejter and P. Horak, {\it A generalization of Lee codes}, Des. Codes Cryptogr., {\bf 70} (2014), 77--90.
\bibitem{B1} M. Behzad, {\it Graphs and their chromatic numbers}, PhD thesis, Michigan State University, 1965.

\bibitem{B2} M. Behzad, {\it The total chromatic number}, Proc. Conf. Combin. Math. and Appl. (1969), 1--8. 

\bibitem{BV} V. Boltyanski, H. Martini and P. S. Soltan, Excursions into Combinatorial Geometry, Springer, 1997.

\bibitem{Bondy} J. A. Bondy and U. S. R. Murty, Graph Theory, Springer, 2008.

\bibitem{Mazzu} S. Dantas, C. M. H. de Figueiredo, G. Mazzuocollo, M. Preissmann, V. F. dos Santos, D. Sasaki, {\it On the total coloring of generalized Petersen graphs}, Discrete Math., {\bf 339} (2016), 1471--1475.

\bibitem{+1} I. J. Dejter, {\it Total coloring of regular graphs of girth = degree + 1}, Ars Combinatoria, {\bf 162} (2025), 159--176.

\bibitem{worst} I. J. Dejter, {\it Worst-case efficient dominating sets in digraphs}, Discrete Appl. Math, {\bf 161} (2003), 944-952.


\bibitem{PDS} I. J. Dejter, {\it Perfect domination in regular grid graphs}, Australasian J. of Combinatorics, {\\bf 42} (2008). 99--114.

\bibitem{Delgado} I. J. Dejter and A. Delgado, {\it Perfect domination in rectangular grid graphs}, Jour. Combin. Math. and Combin. Comput., {\bf 070} (2009), 177--196.

\bibitem{D73} I. J. Dejter and O. Serra, {\it Efficient dominating sets in Cayley graphs}, Discrete Appl. Math., {\bf 129} (2003), 319--328.

\bibitem{Tomai} I. J. Dejter and O. Tomaiconza, {\it Nonexistence of efficient dominating sets in the Cayley graphs generated by transposition trees of diameter 3}, Discrete Appl. Math, {\bf232} (2017), 116--124.

\bibitem{Deng} Y-P. Deng, Y.Q. Sun, Q. Liu and Y-C. Wang, {\it Efficient domination sets in circulant graphs}, Discrete Mathematics, {\bf 340} (2017), 1503--1507. 


\bibitem{Feng} Y. Feng, W. Lin, {\it A concise proof for total coloring subcubic graphs}, Inform. Process. Lett., {\bf 113} (2013), 664--665.

\bibitem{tc-as} J. Geetha, N. Narayanan and K. Somasundaram, {Total colorings-a survey}, AKCE Int. Jour. of Graphs and Combin., {\bf 20}, (2023), issue 3. 339--351. 


\bibitem{GrossT} J. L. Gross and T. W. Tucker, Topological Graph Theory, Wiley 1987.

\bibitem{Branko} B. Gr\"unbaum and G. C. Sheppard, Tilings and Patterns, W. H. Freeman, New York, 1987.

\bibitem{EDS} T. Haynes, S. T. Hedetniemi and M. A. Henning, {\it Efficient Domination in Graphs}, in Domination in Graphs: Core Concepts, Springer Monographs in Mathematics, Springer.

\bibitem{Knor} M. Knor and P. Poto\v{c}nik, {\it Efficient domination in vertex-transitive graphs}, Eur. Jour. Combin., {\bf 33} (2012), 1755--1764.


\bibitem{Rosen} M. Rosenfeld, {\it On the total chromatic number of a graph}, Israel J. Math., {\bf 9} (1971), 396--402.

\bibitem{Arroyo} A. S\'anchez-Arroyo, {\it Determining the total coloring number is NP-hard}, Discrete Math., {\bf 78} (1979), 315--319.

\bibitem{soifer} A. Soifer, The New Mathematical Coloring Book, 2nd ed., Springer-Verlag, 2024,


\bibitem{Vi} N. Vijayaditya, {\it On total chromatic number of a graph}, J. London Math. Soc., {\bf 2} (1971), 405--408.

\bibitem{V} V. G. Vizing, {\it On an estimate of the chromatic class of a $p$-graph}, Discret Analiz, {\bf 3} (1969), 25--30.


\end{thebibliography}
\end{document}